\documentclass[reqno, final,a4paper]{amsart}
\usepackage{color}
\usepackage[colorlinks=true,allcolors=blue
]{hyperref}
\usepackage{amsmath, amssymb, amsthm}
\usepackage{mathrsfs}
\usepackage{mathtools}
\usepackage[linesnumbered,ruled,vlined]{algorithm2e}

\SetCommentSty{mycommfont}
\SetKwInOut{Input}{input}\SetKwInOut{Output}{output}

\usepackage[noadjust]{cite}
\usepackage[noabbrev,capitalize,nameinlink]{cleveref}
\crefname{equation}{}{}
\usepackage{fullpage}
\usepackage{graphics}
\usepackage{tikzscale}
\usepackage{pifont}
\usepackage{tikz}
\usetikzlibrary{fit, backgrounds, shapes, calc}
\usepackage{pgfplots}
\pgfplotsset{compat=1.12}
\usepackage{bbm}
\usepackage[T1]{fontenc}
\usetikzlibrary{arrows.meta}

\usepackage{environ}
\usepackage{framed}
\usepackage{url}

\usepackage[noend]{algpseudocode}
\usepackage[labelfont=bf]{caption}
\usepackage{framed}
\usepackage[framemethod=tikz]{mdframed}
\usepackage{appendix}
\usepackage{graphicx}
\usepackage[textsize=tiny]{todonotes}
\usepackage{tcolorbox}
\usepackage{xcolor}
\usepackage[shortlabels]{enumitem}
\crefformat{enumi}{#2#1#3}
\crefrangeformat{enumi}{#3#1#4 to~#5#2#6}
\crefmultiformat{enumi}{#2#1#3}
{ and~#2#1#3}{, #2#1#3}{ and~#2#1#3}
\allowdisplaybreaks

\usepackage{ifthen}
\numberwithin{equation}{section}
\newtheorem{theorem}{Theorem}[section]
\newtheorem{proposition}[theorem]{Proposition}
\newtheorem{lemma}[theorem]{Lemma}

\crefname{claim}{Claim}{Claims}

\newtheorem*{question*}{Question}

\theoremstyle{definition}
\newtheorem{definition}[theorem]{Definition}

\newtheorem*{definition*}{Definition}

\newtheorem{fact}[theorem]{Fact}
\newtheorem*{fact*}{Fact}

\crefname{fact}{Fact}{Facts}

\theoremstyle{remark}

\newcommand{\R}{\mathbb{R}}

\renewcommand{\Pr}{\mathbb{P}}
\newcommand{\E}{\mathbb{E}}

\def\calE{\mathcal{E}}
\def\calF{\mathcal{F}}

\def\calS{\mathcal{S}}

\def\epsilon{\varepsilon}

\newcommand\xx{\boldsymbol{\mathit{x}}}

\let\originalleft\left
\let\originalright\right
\renewcommand{\left}{\mathopen{}\mathclose\bgroup\originalleft}
\renewcommand{\right}{\aftergroup\egroup\originalright}

\definecolor{ForestGreen}{rgb}{0.0333,0.4451,0.0333}
\definecolor{DarkRed}{rgb}{0.65,0,0}
\definecolor{NavyBlue}{rgb}{0,0,0.5}

\title{Counting Perfect Matchings in Dirac Hypergraphs}
\author[Kwan]{Matthew Kwan}

\address{Institute of Science and Technology Austria (ISTA).}
\email{matthew.kwan@ist.ac.at}

\author[Safavi]{Roodabeh Safavi}

\address{Institute of Science and Technology Austria (ISTA)}
\email{roodabeh.safavi@ist.ac.at}

\author[Wang]{Yiting Wang}

\address{Institute of Science and Technology Austria (ISTA).}
\email{yiting.wang@ist.ac.at}

\thanks{
This project has received funding from the European Research Council (ERC), via grant agreements ``MoDynStruct'' No.\ 101019564 and ``RANDSTRUCT'' No.\ 101076777. In particular, the former grant agreement is under the European Union's Horizon 2020 research and innovation programme \includegraphics[width=4.5mm, height=3mm]{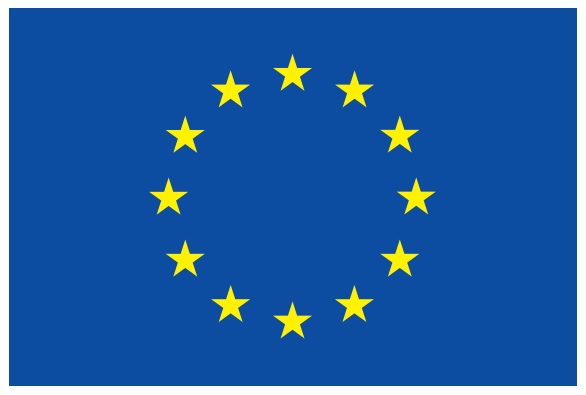}. This work was further supported by the Austrian Science Fund (FWF) and netIDEE SCIENCE project P 33775-N.
For Open Access purposes, the authors have applied a CC BY public copyright license to any author
accepted manuscript version arising from this submission.}
\begin{document}

\begin{abstract}
One of the foundational theorems of extremal graph theory is\emph{
Dirac's theorem}, which says that if an $n$-vertex graph $G$ has
minimum degree at least $n/2$,
then $G$ has a Hamilton cycle, and therefore a perfect matching (if
$n$ is even). Later work by S\'arkozy, Selkow and Szemer\'edi showed
that in fact Dirac graphs have \emph{many} Hamilton cycles and perfect
matchings, culminating in a result of Cuckler and Kahn that gives
a precise description of the numbers of Hamilton cycles and perfect
matchings in a Dirac graph $G$ (in terms of an entropy-like parameter
of $G$).

In this paper we extend Cuckler and Kahn's result to perfect matchings
in hypergraphs. For positive integers $d<k$, and for $n$ divisible
by $k$, let $m_{d}(k,n)$ be the minimum $d$-degree that ensures
the existence of a perfect matching in an $n$-vertex $k$-uniform hypergraph.
In general, it is an open question to determine (even asymptotically)
the values of $m_{d}(k,n)$, but we are nonetheless able to prove
an analogue of the Cuckler--Kahn theorem, showing that if an
$n$-vertex $k$-uniform hypergraph $G$ has minimum $d$-degree at least
$(1+\gamma)m_{d}(k,n)$ (for any constant $\gamma>0$), then the number
of perfect matchings in $G$ is controlled by an entropy-like parameter
of $G$. This strengthens cruder estimates arising from work of Kang--Kelly--K\"uhn--Osthus--Pfenninger and Pham--Sah--Sawhney--Simkin. 
\end{abstract}
\maketitle

\section{Introduction}

\subsection{Counting perfect matching in Dirac graphs}
\emph{Dirac's theorem} is one of the foundational theorems in extremal
graph theory. It says that if an $n$-vertex graph $G$ has minimum
degree at least $n/2$, then $G$ has a Hamilton cycle, and therefore
has a perfect matching if $n$ is even (taking every second edge of the
Hamilton cycle). A huge number of variants and extensions of Dirac's
theorem have been considered over the years; in particular, one important
direction of research (seemingly first proposed by Bondy~\cite[p.~79]{Bondy-book})
is to \emph{count} Hamilton cycles and perfect matchings, given the
same minimum degree condition. As one of the first applications of
the so-called \emph{regularity method}, S\'arkozy, Selkow and Szemer\'edi~\cite{SSS}
resolved this problem up to an $\exp(O(n))$ error term: every graph
on $n$ vertices with minimum degree at least $n/2$ has at least $n^{n}\exp(-O(n))$
Hamilton cycles and at least $n^{n/2}\exp(-O(n))$ perfect matchings.

A much sharper estimate was later proved in influential work of Cuckler
and Kahn~\cite{CK-entropy,CK-counting}: they managed to resolve this problem up to an
$\exp(o(n))$ error term, characterising the number of Hamilton cycles and perfect matchings
in terms of a parameter they called \emph{graph entropy}. We state the perfect matching case of their result, after introducing some notation and definitions.

\begin{definition}[Fractional perfect matchings]\label{def:fractional-PM}
A \emph{fractional perfect matching} $\xx$ of $G$ is an assignment of a weight $\xx[e] \in [0,1]$ to each edge of $G$, such that $\sum_{e\ni v}\xx[e] = 1$ for each vertex $v$ in $G$.
\end{definition}

\begin{definition}[Graph entropy]\label{def:graph-entropy}
    For a fractional perfect matching $\xx$ of a graph $G$, the \emph{entropy} of $\xx$ is defined as
    \[
    h(\xx) \coloneqq \sum_{e\in E(G)}\xx[e] \ln(1/\xx[e])\,,
    \]
    where we use the convention $0\ln (1/0)=0$. The entropy of $G$ is defined as
    \[
    h(G) \coloneqq \sup_{\xx} h(\xx)\,,
    \]
    where the supremum is taken over all fractional perfect matchings in $G$. 
\end{definition}
\begin{theorem}[\cite{ CK-entropy,CK-counting}]\label{thm:CK-1}
    Let $G$ be an $n$-vertex graph with minimum degree at least $n/2$. Then the number of perfect matchings in $G$ is
    \[
        \mathrm{e}^{h(G)}(1/\mathrm{e}+o(1))^{n/2},
    \]
    where asymptotics are as $n\to \infty$ (along a sequence of even integers).
\end{theorem}

Note that $h(G)$, as the solution to a convex optimisation problem, can be efficiently computed.

In addition to proving \cref{thm:CK-1}, Cuckler and Kahn also studied the minimum possible value of $h(G)$ among all graphs with specified minimum degree, and as a result they managed to deduce the following rather appealing theorem: for a given minimum degree (exceeding $n/2$), a \emph{random} graph of the appropriate density has essentially the minimum possible number of perfect matchings.

\begin{theorem}[\cite{CK-entropy, CK-counting}] \label{thm:CK-2}
Let $G$ be an $n$-vertex graph with minimum degree at least $pn$, for some $p\in [1/2,1]$. Then the number of perfect matchings in $G$ is at least 
    \[
        \Phi(K_n)\cdot (p+o(1))^{n/2},
    \]
    where $\Phi(K_n)=n!/((n/2)! 2^{n/2})$ is the number of perfect matchings in the complete graph on $n$ vertices. (Here asymptotics are as $n\to \infty$, along a sequence of even integers).
\end{theorem}
(Note that $\Phi(K_n) \cdot p^{n/2}$ is the expected number of perfect matchings in a random graph on $n$ vertices in which every edge is present with probability $p$. It is not hard to show that such random graphs certify tightness of the bound in \cref{thm:CK-2}).

\subsection{Extensions to hypergraphs}
A \emph{$k$-uniform hypergraph} (which we sometimes abbreviate as a \emph{$k$-graph}) is a generalisation of a graph where each edge is a set of exactly $k$ vertices (so, graphs are 2-uniform hypergraphs). A central program 
in combinatorics is to understand the extent to which classical theorems about graphs can be generalised to hypergraphs, and there has been particular interest in hypergraph generalisations of Dirac's theorem. However, despite intensive effort over the last 20 years, even the most basic questions largely remain unanswered (see the surveys \cite{RR-survey,Zhao-survey} for much more than we will be able to discuss here).

The notion of a perfect matching generalises in the obvious way to hypergraphs (there are various ways to generalise the notion of a Hamilton cycle, too, though we will not consider these in the present paper). The study of Dirac-type theorems for hypergraphs generally concerns the following parameters.
\begin{definition}[Hypergraph degrees]\label{def:min-degrees}
For nonnegative integers $d<k$, the \emph{minimum $d$-degree} $\delta_d(G)$ of an $n$-vertex $k$-uniform hypergraph $G$ is the minimum, over all sets of $d$ vertices $S$, of the degree of $S$ (i.e., the number of edges $e\supset S$).
\end{definition}

\begin{definition}[Hypergraph Dirac threshold]\label{def:dirac-threshold}
     For $n$ divisible
by $k$, let $m_{d}(k,n)$ be the minimum $d$-degree that ensures
the existence of a perfect matching in a $k$-uniform hypergraph on $n$ vertices. Noting that a set of $d$ vertices in an $n$-vertex $k$-uniform hypergraph can have degree at most $\binom{n-d}{k-d}$, let
\[\alpha_d(k) \coloneqq \lim_{n\rightarrow \infty} \frac{m_d(k,n)}{\binom{n-d}{k-d}}\]be the ``asymptotic Dirac threshold'' for $d$-degree in $k$-uniform hypergraphs (it was proved by Ferber and Kwan~\cite[Theorem~1.2]{FM-Dirac-threshold} that this limit exists).
\end{definition}
In general, the values of $\alpha_d(k)$ are unknown: determining them is one of the most important open problems in extremal hypergraph theory (with surprising connections to certain inequalities in probability theory~\cite{AFHRRS}). Apart from sporadic values of $(d,k)$ (discussed in the survey \cite{Zhao-survey}), the most general result is due to Frankl and Kupavskii~\cite{Erdos-matching-conj}, who showed that $\alpha_d(k) = 1/2$ for $d \geq 3k/8$.

Despite the fact that the values of $\alpha_d(k)$ are in general not known, recently Kang, Kelly, K\"uhn, Osthus and Pfenninger~\cite{spread-measure} and Pham, Sah, Sawhney and Simkin~\cite{PSS23} (see also \cite{kelly2023optimal})
were able to prove estimates on the number of perfect matchings above the Dirac threshold, as a consequence of their work on so-called \emph{spread measures}\footnote{Actually, similar estimates are also available using the so-called \emph{weak regularity lemma} for hypergraphs, and the so-called \emph{absorption method} (see \cite{FM-Dirac-threshold} for more on both), though the results in \cite{spread-measure,kelly2023optimal,PSS23} have much better quantitative aspects.}. Specifically, for an $n$-vertex $k$-graph $G$ (with $n$ divisible by $k$) which has minimum $d$-degree at least $(\alpha_d(k)+\gamma)\binom{n-d}{k-d}$ (for some positive $\gamma>0$), the number of perfect matchings in $G$ is
\[n^{(1-1/k)n}\exp(-O(n))\]
(where the implicit constant in the big-Oh notation is allowed to depend on $d$, $k$ and $\gamma$). This can be viewed as an analogue of the S\'arkozy--Selkow--Szemer\'edi bound in the graph case (though the fact that the implicit constant depends on $\gamma$ is not very desirable).

In the particular case where $d=k-1$ (where the Dirac threshold is known exactly, thanks to influential work of R\"odl, Ruci\'nski and Szemer\'edi~\cite{RRS09}), a hypergraph analogue of Cuckler and Kahn's \cref{thm:CK-2} is available. Indeed, it follows from results of Ferber, Hardiman and Mond~\cite{Ferber-counting} (improving earlier bounds of Ferber, Krivelevich and Sudakov~\cite{FKS-counting-hypergraph})
that if a $k$-uniform hypergraph $G$ satisfies $\delta_{k-1}(G)\ge (\alpha_{k-1}(k)+\gamma)(n-(k-1))=(1/2+\gamma)(n-(k-1))$ for any constant $\gamma>0$, then the number of perfect matchings in $G$ is at least
\begin{equation}\Phi(K_n^{(k)}) \cdot (p+o(1))^{n/k}\,,\label{eq:like-random}\end{equation}
where $p=\delta_{k-1}(G)/n$, and $\Phi(K_n^{(k)})=n!/((n/k)! (k!)^{n/k})$ is the number of perfect matchings in the complete $k$-uniform hypergraph on $n$ vertices.

Ferber, Hardiman and Mond also raised the problem of proving similar results for all $d<k$, but it was observed by Sauermann (see \cite[Section~5]{Ferber-counting}) that the particular bound in \cref{eq:like-random} does \emph{not} generalise to all $d$. Specifically, there is a 3-uniform hypergraph $G$ satisfying $\delta_{1}(G)\ge (\alpha_{1}(3)+0.01)(n-1)$ which has exponentially fewer than $\Phi(K_n^{(k)})(\alpha_{1}(3)+0.01)^{n/k}$ perfect matchings.

\subsection{Main results}
In this paper, we extend \cref{thm:CK-1} (i.e., the estimate in terms of graph entropy) to hypergraphs, for all $d<k$, and we deduce a hypergraph analogue of \cref{thm:CK-1} for $d\ge k/2$ (due to the above observation of Sauermann, some assumption on $d$ is necessary for this result).

Note that the definition of graph entropy (\cref{def:graph-entropy}) extends naturally to hypergraphs. First, our hypergraph analogue of \cref{thm:CK-1} is as follows.
\begin{theorem} \label{thm:num-of-PM}
    Fix constants $1\le d<k$ and $\gamma > 0$.
    Let $G$ be an $n$-vertex $k$-uniform hypergraph with minimum d-degree at least $(\alpha_d(k)+\gamma)\binom{n-d}{k-d}$ (recalling \cref{def:dirac-threshold}). Then the number of perfect matchings in $G$ is
    \[
        \mathrm{e}^{h(G)}(1/\mathrm{e}+o(1))^{(1-1/k)n}\,,
    \]
    where asymptotics are as $n\to \infty$ (along a sequence of integers divisible by $k$).
\end{theorem}
We remark that Glock, Gould, Joos, K\"{u}hn and Osthus~\cite{Stefan-counting} explicitly asked if the entropy approach of Cuckler and Kahn can be generalized to hypergraphs; \cref{thm:num-of-PM} answers this question in the setting of perfect matchings.

For $d \geq k/2$, we are able to determine the minimum possible value of $h(G)$ over all $k$-graphs with a given minimum $d$-degree, and we are therefore able to deduce an analogue of \cref{thm:CK-2}, as follows.
\begin{theorem}\label{thm:lb-no-entropy} 
    Fix constants $1\le d<k$ satisfying $d\ge k/2$, and fix a constant $\gamma > 0$.
    Let $G$ be an $n$-vertex $k$-uniform hypergraph with minimum $d$-degree $\delta_d(G)\ge (\alpha_d(k)+\gamma)\binom{n-d}{k-d}$ (recalling \cref{def:dirac-threshold}), and let $p=\delta_d(G)/\binom{n-d}{k-d}$.    
    Then the number of perfect matchings in $G$ is at least 
    \[
        \Phi(K_n^{(k)})\cdot (p+o(1))^{n/k},
    \]
    where $\Phi(K_n^{(k)})=n!/((n/k)! (k!)^{n/k})$ is the number of perfect matchings in the complete $k$-uniform hypergraph on $n$ vertices, and asymptotics are as $n\to \infty$ (along a sequence of integers divisible by $k$).
\end{theorem}
As mentioned in the last subsection, the $d=k-1$ case of \cref{thm:lb-no-entropy} was already known, thanks to work of Ferber, Hardiman and Mond~\cite{Ferber-counting}.

Our proof approach follows the same approximate strategy as Cuckler and Kahn~\cite{CK-entropy,CK-counting}. Indeed, \cref{thm:num-of-PM} really consists of a lower bound and an upper bound, proved separately. The lower bound is proved via analysis of a so-called \emph{random greedy process} (as in \cite{CK-counting}), and the upper bound is proved using the so-called \emph{entropy method} (as in \cite{CK-entropy}). Actually, the desired upper bound can be quite easily deduced from a general estimate proved by Kahn~\cite{Kahn-Shamir} in his work on \emph{Shamir's conjecture}; 
the main content of this paper is the lower bound.

The most significant point of departure from \cite{CK-counting} is that in the hypergraph setting it seems to be harder to study properties of the maximum-entropy fractional perfect matching. To give some context: the proof of \cref{thm:num-of-PM} considers a random process guided by a fractional perfect matching $\xx$, and to maintain good control over the trajectory of this process it is important that $\xx$ is quite evenly distributed over the edges of $G$.

In \cite{CK-entropy}, a simple weight-shifting argument is used to show that the maximum-entropy fractional perfect matching is always quite evenly distributed. Indeed, if $\xx$ has ``unevenly distributed weights'', one can shift weights around a 4-cycle to ``even out the weights'' and increase the entropy. There does not seem to be a direct generalisation of this fact to hypergraphs: in the hypergraph setting we cannot rule out the possibility that the entropy-maximising fractional perfect matching $\xx^\star$ has a few very high-weight edges (these can be viewed as ``pockets of trapped energy''). So, we introduce an additional ``annealing'' technique: we slightly perturb $\xx^\star$ using a \emph{spread measure} on perfect matchings (obtained from \cite{spread-measure,kelly2023optimal}) to ``release the trapped energy'': namely, such a perturbation enables entropy-increasing weight-shifting operations. The resulting fractional perfect matching might not be entropy-maximising, but it is \emph{nearly} entropy-maximising, which is good enough.
Similar ideas also appeared in a recent work of Joos and Schrodt~\cite{Count-oriented-trees} on counting oriented trees in digraphs with large semidegree.

\subsection{Further directions}
Perhaps the most important problem left open by our work is to prove a result similar to \cref{thm:lb-no-entropy} for $d < k/2$. As mentioned in the introduction, the specific bound in \cref{thm:lb-no-entropy} cannot hold for all $d<k/2$, due to a construction of Sauermann (see \cite[Section~5]{Ferber-counting}) but we can still ask for the minimum possible number of perfect matchings given a $d$-degree condition. This seems to have been first explicitly asked by Ferber, Hardiman and Mond~\cite{Ferber-counting}. Given \cref{thm:num-of-PM}, this problem now comes down to an entropy-maximisation problem over fractional perfect matchings.

Second, there is the question of improving the quantitative aspects of \cref{thm:num-of-PM,thm:lb-no-entropy}. We have presented both theorems with a multiplicative error of the form $(1+o(1))^n$, but the proof really gives a multiplicative error of the form $(1+O(n^{-c_k}))^n$ for some $c_k\ge 0$. 
It is unclear what exactly is the sharpest bound one can hope for, without using more refined information about our hypergraph $G$ than its entropy. 

Third, the results in this paper only apply when the Dirac threshold is ``asymptotically exceeded'' (i.e., when the minimum $d$-degree is at least $(1+\gamma)m_d(k,n)$ for some constant $\gamma>0$). It would be interesting to prove similar results under only the assumption $d\ge m_d(k,n)$. In particular, this may be within reach in the case $d\ge k/2$, as in this case the exact value of $m_d(k,n)$ is known (thanks to Treglown and Zhao~\cite{TZ12,TZ13}).

Finally, we would like to ask if \cref{thm:num-of-PM,thm:lb-no-entropy} can be extended to other types of spanning substructures other than perfect matchings. In particular, there are several different notions of Hamilton cycles in hypergraphs (``tight cycles'', ``loose cycles'', ``Berge cycles''...), which would be a good starting point. There has actually already been quite some work in this direction, but only in the case $d=k-1$ (i.e., where we impose a condition on $\delta_{k-1}(G)$). In particular, the work of Ferber, Hardiman and Mond~\cite{Ferber-counting}, mentioned in the introduction, proves an analogue of the $d=k-1$ case of \cref{thm:lb-no-entropy} for certain types of hypergraph Hamilton cycles.

\subsection{Notation}
As convention, we abbreviate a $k$-uniform hypergraph as a $k$-graph. We also call a subset of vertices of size $d$ a $d$-set. 

Our graph-theoretic notation is for the most part standard. For a graph/hypergraph $G$, we write $V(G)$ and $E(G)$ for its sets of vertices and edges, respectively. 
For $S\subseteq V(G)$, we write $\deg_{G}(S)$ to denote its degree.
If $G$ is a graph, we write $\delta(G)$ to denote the minimum degree of $G$. If $G$ is a $k$-graph, for $0\leq d\leq k-1$, we write $\delta_d(G)$ to denote its minimum $d$-degree. 
We implicitly treat an edge as a set of vertices.
We say a bipartite graph on $2n$ vertices is \emph{balanced} if each part has size $n$.

Our use of asymptotic notation is standard as well. For functions $f=f(n)$ and $g=g(n)$, we write $f=O(g)$ to mean that there is a constant $C$ such that $|f(n)|\le C|g(n)|$ for all $n$, $f=\Omega(g)$ to mean that there is a constant $c>0$ such that $f(n)\ge c|g(n)|$ for sufficiently large $n$, and $f=o(g)$ or $g=\omega(f)$ to mean that $f/g\to0$ as $n\to\infty$.
Slightly less standardly, for $a, b\in \R$ we write $a = \pm b$ to mean $a \in [-b,b]$.

Finally, we also use standard probabilistic notation: $\Pr[\mathcal E]$ denotes the probability of an event $\mathcal E$ and $\E[X]$ denotes the expected value of a random variable $X$. We say a sequence of events $(\mathcal E_n)_{n\in \mathbb N}$ happens with high probability (abbreviated as whp) if $\lim_{n\rightarrow \infty}\Pr(\calE_n) = 1$.

\subsection{Structure of the paper}
In \cref{sec:prelim} we start with some general preliminaries, before moving on to the proofs of \cref{thm:num-of-PM,thm:lb-no-entropy}.

\cref{thm:num-of-PM} really consists of a lower bound and an upper bound on the number of perfect matchings, which are proved separately. Most of the paper is devoted to the lower bound. Indeed, in \Cref{section:construct-PM}, we will show that in the setting of \cref{thm:num-of-PM} we can always find a fractional perfect matching which simultaneously has almost maximum entropy, and which satisfies an ``approximate uniformity'' condition. 
In \Cref{section:random-greedy}, we will show how to use such a fractional perfect matching to prove the lower bound in \Cref{thm:num-of-PM}, via analysis of a random matching process.

Then, in \Cref{section:entropy-ub}, we deduce the upper bound in \Cref{thm:num-of-PM} from a result of Kahn~\cite{Kahn-Shamir} (proved using the so-called \emph{entropy method}).

Finally, in \Cref{section:lower-bound-for-upper range}, we deduce \Cref{thm:lb-no-entropy} by constructing an explicit fractional perfect matching and lower bounding its entropy. The primary ingredient here is an algebraic construction of fractional perfect matchings in bipartite graphs, due to Cuckler and Kahn~\cite{CK-counting}.

\subsection*{Acknowledgments}We would like to thank the referees for a number of helpful comments and suggestions, which have substantially improved the paper.

\section{Preliminaries}\label{sec:prelim}
First, it is convenient to introduce some notation that will be used throughout the paper.
\begin{definition}[Dirac hypergraphs]\label{def:d-Dirac-k-graph}
Fix $1\leq d \leq k-1$ and $\gamma > 0$, and let $G$ be a $k$-uniform hypergraph on $n$ vertices. We say $G$ is a \emph{$(d,\gamma)$-Dirac $k$-graph} if it satisfies $\delta_d(G) \geq (\alpha_d(k)+\gamma)\binom{n-d}{k-d}$ (recalling \cref{def:dirac-threshold}) and $k\mid n$.
\end{definition}
\begin{definition}
    For a hypergraph $G$, let $\Phi(G)$ be the number of perfect matchings in $G$.
\end{definition}
\begin{definition}[Well-distributed fractional perfect matchings] \label{def:well-distributed}
    For a $k$-graph $G$ an edge-weighting $\xx$ of $G$, and $D \geq 1$, we say that $\xx$ is \emph{$D$-well-distributed} if \[\frac1{D n^{k-1}}\le \xx[e] \le \frac D{n^{k-1}}\] for all edges $e$.
\end{definition}

\subsection{Inequalities on (fractional) perfect matchings}
We now collect some basic inequalities that will be used throughout the proof. First, the following inequality can be proved with an easy double-counting argument.
\begin{fact}[Monotonicity of minimum degrees]
\label{cla:monotone}
    Let $G$ be an arbitrary $k$-uniform hypergraph. Then,
    \[
        \frac{\delta_0(G)}{\binom{n}{k}} \geq \cdots \geq \frac{\delta_{k-1}(G)}{\binom{n-k+1}{1}}\,.
    \]
\end{fact}
\begin{proof}
    Let $0\leq i\leq j\leq k-1$. Fix an arbitrary $S\subseteq V$ of size $i$. Then, $\deg(S) \geq \binom{n-i}{j-i}\delta_j(G)/\binom{k-i}{j-i}$ because we may first extend $S$ to an arbitrary set of size $j$, observe that this extended set is contained in at least $\delta_j(G)$ edges, and then divide by the times an edge is overcounted, which is at most $\binom{k-i}{j-i}$.
    Since $\binom{n-i}{j-i}/\binom{k-i}{j-i} = \binom{n-i}{k-i}/\binom{n-j}{k-j}$, this implies the desired conclusion.
\end{proof}
\cref{cla:monotone} implies that $\alpha_0(k)\geq \alpha_1(k)\geq \dots\geq \alpha_{k-1}(k)$. Since $\alpha_{k-1}(k)$ is known to be equal to $1/2$ (see~\cite{RRS09}), we have the following corollary.
\begin{fact}\label{fact:dirac-bound}
For any $d<k$ we have $\alpha_d(k)\ge 1/2$.
\end{fact}

Recall Jensen's inequality:
\begin{fact}[Jensen's inequality]\label{fact:jensen-inequality}
    Let $\alpha_1,\dots, \alpha_n\in [0,1]$ with $\sum_{i=1}^n \alpha_i = 1$. Let $f: \mathbb{R}\rightarrow \mathbb{R}$ be a convex function. Then, for any $x_1,\dots, x_n\in \mathbb R$, we have $f(\sum_{i=1}^n \alpha_i x_i)\leq \sum_{i=1}^n \alpha_i f(x_i)$. If $f$ is a concave function, then the reverse inequality holds.
\end{fact}

Next, the following inequalities are easy corollaries of Jensen's inequality. 
\begin{lemma}\label{cla:jensen-entropy}
    Let $G$ be a $k$-graph and let $\xx$ be a fractional perfect matching in $G$. Then,
    \[h(\xx) \leq (1-1/k)n\ln n\,.\]
    In addition, if $\xx[e]\leq L$ for every edge $e$, then
    \[
        h(\xx)\geq \frac{n}{k} \ln\biggl(\frac{n}{L^2 k |E(G)|}\biggr)\,.
    \]    
\end{lemma}
\begin{proof}
    Because $\xx$ is a fractional perfect matching, we know that $\sum_{e\ni v}\xx[e] = 1$.
    Applying \Cref{fact:jensen-inequality} to each $v\in V$ with the concave function $t\mapsto \ln t$, $n = |\{e:e\ni v\}|$, $\alpha_i = \xx[e]$ and $x_i = 1/\xx[e]$, we conclude
    \begin{align*}
        h(\xx) =\frac{1}{k} \sum_{v\in V}\sum_{e\ni v} \xx[e] \ln(1/\xx[e]) &\leq \frac{1}{k}\sum_{v\in V}\ln\bigl(|\{e\in E(G): e\ni v\}|\bigr)\\
        &\leq \frac{n}{k} \ln\biggl(\binom{n-1}{k-1}\biggr)\le \frac{n}{k}\ln (n^{k-1})=
        (1-1/k)n\ln n\,.
    \end{align*}

    For the lower bound, applying~\Cref{fact:jensen-inequality} with the convex function $t\mapsto -\ln(t)$,$n = |E(G)|$, $\alpha_i = k\xx[e]/n$ and $x_i = \xx[e]$, we get 
    \[
        h(\xx) = -\sum_{e\in E(G)} \xx[e] \ln(\xx[e]) = -\frac{n}{k} \sum_{e\in E(G)} \frac{k\xx[e]}{n} \ln(\xx[e]) \geq -\frac{n}{k}\ln\biggl(\frac{k}{n}\sum_{e\in E(G)}\xx[e]^2\biggr)\,.
    \]
    Using the assumption that each $\xx[e] \leq L$, and rearranging finishes the proof. 
\end{proof}

\subsection{Entropy}
We will need the notion of entropy (of a random variable), and some of its properties. These can be found in any source on information theory, e.g.\ \cite{CT-entropy}. 
Recall that $\ln$ is the natural logarithm (base-$e$), and write $\operatorname{supp}(X)$ for the support of a random variable $X$
\begin{definition}\label{def:binary-entropy}
    Let $X$ be a random variable with finite support. Then the \emph{entropy} of $X$ is defined as:
    \[
        H(X) \coloneqq \sum_{x\in \operatorname{supp}(X)} \Pr[X= x]\cdot \ln(1/\Pr[X=x])\,.
    \]
    For an event $\mathcal E$ on the same probability space as $X$, we write $H(X\,|\,\mathcal E)$ for the entropy of $X$ in the conditional probability space given $\mathcal E$. Also, for a second random variable $Y$ with finite support, also on the same probability space, the \emph{conditional entropy of $X$ given $Y$} is defined as
    \[
        H(X\mid Y) \coloneqq \sum_{y\in \operatorname{supp}(Y)} \Pr[Y=y] H(X\,|\,\{Y=y\})\,.
    \]
    (Here, for an \emph{event} $\mathcal E$, we write $H(X\,|\,\mathcal E)$ for the entropy of $X$ in the conditional probability space given $\mathcal E$).
\end{definition}
Dropping conditioning cannot decrease the entropy:
\begin{fact}\label{cla:drop-conditioning}
    For any random variables $X,Y$ with finite support, we have $H(X\mid Y)\leq H(X)$.
\end{fact}
The entropy is maximised by a uniform distribution on its support:
\begin{fact}\label{cla:unif-maximizes}
    Let $X$ be a random variable with finite support. Then $H(X)\le \ln(|\operatorname{supp}(X)|)$, with equality when $X$ is a uniform distribution.
\end{fact}
We will also need the chain rule for entropy, as follows.
\begin{fact}\label{cla:chain-rule}
    Let $X_1,\dots, X_n$ be random variables on the same probability space. Then
    \[
        H(X_1,\dots, X_n) = \sum_{i=1}^n H(X_i\mid X_1,\dots, X_{i-1})\,.
    \]
\end{fact}

\section{\texorpdfstring{$D$}{D}-well-distributed fractional perfect matchings with near-maximal entropy}
\label{section:construct-PM}
In this section, for any $(d,\gamma)$-Dirac $k$-graph $G$ we show how to construct an well-distributed fractional perfect matching with almost maximum entropy. Specifically, if $D$ is sufficiently large in terms of $\varepsilon>0$ (and $k$, and $\gamma$), we show that there is a $D$-well-distributed fractional perfect matching $\xx$ satisfying $h(G) - h(\xx) = \varepsilon n$. (Recall the definition of a Dirac hypergraph from \cref{def:d-Dirac-k-graph}, and the definition of well-distributedness from \cref{def:well-distributed}).
\begin{lemma} \label{lem:exist-normal-entropy}
Fix constants $1\le d\le k-1$ and $\gamma>0$, and consider a $(d,\gamma)$-Dirac $k$-graph $G$. For any $\varepsilon>0$ which is sufficiently small in terms of $k,\gamma$, there is a $O(\varepsilon^{-3k})$-well-distributed fractional perfect matching $\xx$ satisfying $h(G) - h(\xx) \le \varepsilon n$.
\end{lemma}
We emphasise that in \cref{lem:exist-normal-entropy} (and in fact throughout the paper), $d,k,\gamma$ are treated as constants (that is to say, the implicit constant in big-Oh notation is allowed to depend on $d,k,\gamma$).

The construction to prove \cref{lem:exist-normal-entropy} consists of two steps. In the first step, we construct a fractional perfect matching $\hat{\xx}$ that is $O(1)$-well-distributed (but whose entropy may be quite far from the maximum).
In the second step, we use $\hat{\xx}$ to ``anneal'' the entropy-maximising perfect matching $\xx^\star$, which enables the use of a weight-shifting argument to obtain the desired fractional perfect matching.

\subsection{\texorpdfstring{$D$}{D}-well-distributed fractional perfect matchings}\label{subsec:normal}
As stated above, the first step is to find a $C$-well-distributed fractional perfect matching $\hat{\xx}$, as follows.
\begin{lemma} \label{lem:exist-normal}
    Fix constants $1\leq d\leq k-1$ and $\gamma > 0$, let $k\mid n$, and let $G$ be an $n$-vertex $(d,\gamma)$-Dirac $k$-graph.
    Then, $G$ has a $O(1)$-well-distributed fractional perfect matching $\hat{\xx}$.
\end{lemma}
We remark that the $d=k-1$ case of \cref{lem:exist-normal} was proved by Glock, Gould, Joos, K\"uhn and Osthus~\cite[Lemma 4.1]{Stefan-counting}, using a switching argument. It seems that this proof approach is also suitable for \cref{lem:exist-normal}, but given recent developments, it is more convenient to deduce \cref{lem:exist-normal} from a so-called \emph{spread measure} independently constructed by Kang, Kelly, K\"{u}hn, Osthus and Pfenninger~\cite{spread-measure} and Pham, Sah, Sawhney and Simkin~\cite{PSS23}. 
Specifically, the following theorem is a corollary of \cite[Theorem~1.5]{spread-measure}\footnote{In the language of \cite[Theorem~1.5]{spread-measure}, we take $s=0$, and use the definition of spreadness for a single edge.}.
\begin{theorem}\label{thm:spread}
Fix constants $1\leq d\leq k-1$ and $\gamma > 0$, let $k\mid n$, and let $G$ be an $n$-vertex $(d,\gamma)$-Dirac $k$-graph. Then, there is a probability distribution on the set of perfect matchings of $G$, such that if a random perfect matching $M$ is sampled from this distribution, we have
    $\Pr[e\in M] \le O(1/n^{k-1})$ for all edges $e$ of $G$.
\end{theorem}
\begin{proof}[Proof of \Cref{lem:exist-normal}]
Throughout this proof, we may assume that $n$ is large in terms of $k,\gamma$. 

Note that any random perfect matching $M$ naturally gives rise to a fractional perfect matching $\xx$, taking $\xx[e] \coloneqq \Pr[e\in M]$. \cref{thm:spread} already gives us a random perfect matching $M$ with the desired upper bound on $\Pr[e\in M]$; we just need to make an adjustment to $M$ to guarantee the desired lower bound.

Specifically, we define our random perfect matching $M$ by combining \cref{thm:spread} with a few rounds of a random greedy matching process, as follows.
\begin{itemize}
    \item Set $G_0=G$ and $M_0=\emptyset$, let $\beta = \gamma/(10 k^2)$ and let $T=\lfloor\beta n\rfloor$.
        \item For $t\in \{1,\dots,T\}$, choose a uniformly random edge $e$ in $G_{t-1}$, and add it to $M_{t-1}$ to form $M_t$. Then delete all vertices of $e$ from $G_{t-1}$ to form $G_t$.
        \item Note that \[\delta_d(G_T)\ge (\alpha_d(k)+\gamma)\binom{n-d}{k-d}-k\cdot \beta n\cdot \binom{n-d-1}{k-d-1}\ge (\alpha_d(k)+\gamma/2)\binom{n-d}{k-d},\] because for a set $S$ of $d$ vertices, deleting any other vertex can reduce the degree of $S$ by at most $\binom{n-d-1}{k-d-1}$. So, $G_T$ is $(d,\gamma/2)$-Dirac, and we can apply \Cref{thm:spread} to obtain a random perfect matching of $G_T$. Add this to $M_T$ to obtain a random perfect matching $M$ of $G$.
    \end{itemize}
    The number of edges in $G_t$, for each $t\le T$, is at least
\begin{equation}(\alpha_d(k)+\gamma)\binom nk-k\cdot\beta n\cdot \binom{n-1}{k-1}\ge \frac12\binom nk\,,\label{eq:uniform-removed-edges}\end{equation}
recalling \cref{fact:dirac-bound}.
    So, for each edge $e$, we have
    \[\Pr[e\in M]=\Pr[e\in M_T]+\Pr[e\in M\setminus M_T]\le \frac{T}{\binom nk/2}+O(1/n^{k-1})=O(1/n^{k-1})\,.\]
    It remains to prove a corresponding lower bound  $\Pr[e\in M]=\Omega_{k,\gamma}(1)$.
    
    For $1\leq t\leq T$, let $\calE_t$ be the event that $e$ is selected in the $t$-th round (i.e., $e\in M_t\setminus M_{t-1}$), and let $\calF_t$ the event that $e\in G_{t-1}$ (i.e., none of the edges intersecting $e$ have been selected by round $t$).
    By definition, the events $\calE_1 \cap \calF_1,\dots, \calE_T\cap \calF_T$ are disjoint, so
    \[\Pr[e\in M]\ge \Pr[e\in M_T] = \Pr[\calF_1 \cap \calE_1] + \cdots + \Pr[\calF_T \cap \calE_T]\,.\]
So, recalling that $T=\Omega_{k,\gamma}(n)$, it suffices to prove that $\Pr[\calF_t \cap \calE_t]=\Pr[\calE_t\mid \calF_t] \Pr[\calF_t]$ is at least $\Omega_k(n^{-k})$, for each $t\le T$. To see this, note that $G$ has at most $k\cdot \binom{n-1}{k-1}$ edges intersecting $e$, and recall \cref{eq:uniform-removed-edges}, so
\[\Pr[\calF_t]\ge 1-t\cdot \frac{k\cdot \binom{n-1}{k-1}}{\binom nk/2}\ge 1/2\,,\qquad \Pr[\calE_t|\calF_t]\ge \frac{1}{\binom n k}=\Omega_k(n^{-k})\,. \qedhere\]
\end{proof}

Any $O(1)$-well-distributed fractional matching already has entropy which differs by at most $O(n)$ from the maximum entropy, as follows.
\begin{proposition}\label{lem:normal-entropy}
    Fix $k\ge 2$ and let $G$ be a $k$-graph on $n$ vertices. Suppose $\xx$ is a $C$-well-distributed fractional perfect matching in $G$ for some constant $C>0$. Then, $h(G) - h(\xx) \le (2\ln C/k) n$.
\end{proposition}
\begin{proof}
    Since $\xx$ is $C$-well-distributed, we have $\xx[e] \leq C/n^{k-1}$ for all edges $e$.
    Thus, plugging in $L = C/n^{k-1}$ and using the trivial bound $|E(G)|\leq \binom{n}{k}$ in \Cref{cla:jensen-entropy}, we have 
    \[
        h(\xx) 
        \geq (1-1/k)n \ln n - \frac{2\ln C}{k}n\,.
    \]
   On the other hand, the upper bound in \Cref{cla:jensen-entropy} implies $h(G)\leq (1-1/k)n\ln n$. The desired result follows.
\end{proof}

\subsection{Shifting structures}\label{subsec:shifting}
Next, we define the notion of a \emph{shifting structure}, which we will use to ``move weight around'' in a fractional perfect matching. We remark that this same structure was considered by H\`{a}n, Person and Schacht~\cite{Absorbing} in their proof of the so-called \emph{strong absorbing lemma}.
\begin{definition}[Shifting structure]\label{def:shifting-structure}
    Let $k\ge 2$ and consider a $k$-graph $G$.
    Let $e = \{v_1,\dots, v_k\}$ and $f = \{u_1,\dots, u_k\}$ be two edges in $E$ such that $e\cap f = \{v_1\} = \{u_1\}$. A \emph{shifting structure} on $(e,f)$ is a sequence $\mathcal U = (U_2,\dots, U_k)$ of sets of $k-1$ vertices, satisfying the following conditions:
    \begin{itemize}
        \item $U_2,\dots, U_k$ do not share any vertices with $e$ or $f$.
        \item $U_2,\dots, U_k$ are pairwise disjoint. 
        \item For each $2\leq i \leq k$, both $e_i \coloneqq U_i\cup \{u_i\}$ and $f_i\coloneqq U_i \cup \{v_i\}$ form edges of $G$.
    \end{itemize}
\end{definition}
For convenience, we let $e_1=e$ and $f_1=f$. See \cref{fig:swap} for an illustration.
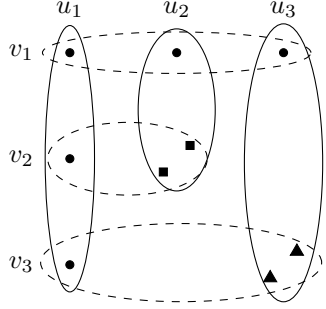
\begin{figure}[ht]
    \centering
    \begin{tikzpicture}    
    \draw[fill=black] (0,0) circle[radius=1.5pt] node(v3) {};
    \draw[fill=black] (0,40pt) circle[radius=1.5pt] node(v2) {};
    \draw[fill=black] (0,80pt) circle[radius=1.5pt] node(v1) {};
    \draw[fill=black] (40pt,80pt) circle[radius=1.5pt] node(u2) {};
    \draw[fill=black] (80pt,80pt) circle[radius=1.5pt] node(u3) {};
    \node(U21)[draw, inner sep = 1.5pt, rectangle, fill=black] at (35pt,35pt) {};
    \node(U22)[draw, inner sep = 1.5pt, rectangle, fill=black] at (45pt,45pt) {};
    \node(U31)[draw, inner sep = 1pt, regular polygon, regular polygon sides=3, fill=black] at (85pt,5pt) {};
    \node(U32)[draw, inner sep = 1pt, regular polygon, regular polygon sides=3, fill=black] at (75pt,-5pt){};

    \node[left=10pt] at (v1) {$v_1$};
    \node[left=10pt] at (v2) {$v_2$};
    \node[left=10pt] at (v3) {$v_3$};
    \node[above=10pt] at (v1) {$u_1$};
    \node[above=10pt] at (u2) {$u_2$};
    \node[above=10pt] at (u3) {$u_3$};

    \node[fit=(v1)(v3), ellipse, draw, inner xsep=3pt, inner ysep=-8pt] {};
    \node[fit=(v1)(u3), ellipse, draw, dashed, inner xsep=-8pt, inner ysep=2pt] {};
    \node[fit=(v2)(U21)(U22), ellipse, draw, dashed, inner xsep=-4pt, inner ysep=3pt] {};
    \node[fit=(u2)(U21)(U22), ellipse, draw , inner xsep=3.5pt, inner ysep=-3.5pt] {};
    \node[fit=(v3)(U31)(U32), ellipse, draw, dashed, inner xsep=-8pt, inner ysep=3pt] {};
    \node[fit=(u3)(U31)(U32), ellipse, draw, inner xsep=3.5pt, inner ysep=-8pt] {};
\end{tikzpicture}
    \caption{We illustrate a shifting structure for $k=3$. The three dotted edges, from top to bottom, are $f_1(= f)$, $f_2$ and $f_3$, and the three solid edges, from left to right, are $e_1 (= e)$, $e_2$ and $e_3$. The set $U_2$ contains the two square-shaped vertices and the set $U_3$ contains the two triangle-shaped vertices.}
    \label{fig:swap}
\end{figure}

We can use a shifting structure to move weight from $e_1,\dots,e_k$ to $f_1,\dots,f_k$, as follows.

\begin{definition}[shift operation]
    Given a $k$-graph $G$ and two edges $e,f\in E(G)$, consider a shifting structure $\mathcal U = (U_2,\dots, U_k)$ on $(e,f)$, and consider a fractional perfect matching $\xx$. For $\Delta\ge 0$ satisfying $\Delta\le \min_i \xx[e_i]$ and $\Delta\le \min_i (1-\xx[f_i])$, let $\operatorname{shift}_\Delta(\xx, \mathcal U)$ be  the fractional perfect matching obtained from $\xx$ by changing $\xx[e_i]$ to $\xx[e_i] - \Delta$ and changing $\xx[f_i]$ to $\xx[f_i] + \Delta$ for $1\leq i\leq k$.
\end{definition}

If a fractional matching is ``imbalanced'' on a shifting structure, in the sense that the product $\xx[e_1]\cdots \allowbreak\xx[e_k]$ is larger than the product $\xx[f_1]\cdots\xx[f_k]$, we can use that imbalance to perform a shifting operation which increases the entropy. In particular, we can do this if $\xx[e_1]$ is much larger than $\xx[f_1],\dots,\xx[f_k]$, and $\xx[e_2],\dots,\xx[e_k]$ are not too small, as follows.

\begin{lemma} \label{lem:weight-shift}
   Consider a $k$-graph $G$ and a fractional perfect matching $\xx$.
    Let $\mathcal U$ be a shifting structure on $(e,f)$. For some $\Delta,\eta>0$, suppose that $\xx[e_i] \geq 2\Delta$ and $\xx[f_i] \leq \eta-\Delta$ for all $i\in \{1,\dots,k\}$. 
    Then
    \[
        h(\operatorname{shift}_\Delta(\xx, \mathcal U))-h(\xx)\ge \Delta \ln \left(\frac{\xx[e_1]\Delta^{k-1}}{2\eta^k}\right)\,.
    \]
\end{lemma}
Note that if $\xx[e_1]\Delta^{k-1}> 2\eta^k$, the right-hand side is positive, meaning that the shift operation increases the entropy.
\begin{proof}
    For $0\le q\le \Delta$, let  $
        f(q)=h(\operatorname{shift}_q(\xx, \mathcal U))
    $.
    Computing the derivative, we see that
    \[
        f^\prime(q) = \ln\left( \frac{(\xx[e_1]-q)\cdots (\xx[e_k]-q)}{(\xx[f_1] + q)\cdots  (\xx[f_k] + q)}\right)\ge \ln\left(\frac{(\xx[e_1]/2) \Delta^{k-1}}{\eta^k}\right)\,,
    \]
    for $q\le \Delta$, and the desired result follows.
\end{proof}

If every vertex in $e\cup f$ has sufficiently large degree, then it is easy to find shifting structures, as follows.
\begin{lemma}\label{lem:find-structure}
Let $G$ be an $n$-vertex $k$-graph, and fix two edges $e = \{v_1,\dots, v_k\}$ and $f = \{u_1,\dots, u_k\}$ intersecting in a single vertex $u_1=v_1$. Suppose each of $u_2,\dots,u_k,v_2,\dots,v_k$ has degree at least $(1/2+\beta)\binom{n-1}{k-1}$. Then, if $n$ is sufficiently large (in terms of $k,\beta$), there is a shifting structure on $(e,f)$.
\end{lemma}
\begin{proof}
A shifting structure is a sequence of sets $U_2,\dots,U_k$; we choose these sets one-by-one in a greedy fashion. Indeed, suppose that we have made a choice for $U_2,\dots,U_{i-1}$. To choose $U_i$, first note that by inclusion-exclusion, there are at least $2\cdot (1/2 + \beta)\binom{n-1}{k-1} - \binom{n}{k-1}\ge \beta\binom{n-1}{k-1}$ sets of $k-1$ vertices forming an edge with both $v_i$ and $u_i$. At most $O_k(n^{k-2})$ of these sets involve other vertices of $e\cup f$ or other vertices of $U_2\cup \dots\cup U_{i-1}$, so there is at least one such set which is a valid choice for $U_i$.
\end{proof}

\subsection{Uniformising the maximum-entropy fractional perfect matching}
Now we use the preparations in \cref{subsec:normal,subsec:shifting} to prove \cref{lem:exist-normal-entropy}. As mentioned in the introduction, in the case $k=2$ (i.e., the graph case), Cuckler and Kahn~\cite{CK-counting} showed that the maximum-entropy fractional matching $\xx^\star$ is already well-distributed \footnote{For the curious reader, this argument appears in \cite[Lemma~3.2]{CK-counting}. (It is stated with an assumption of ``$\xi$-normality'', which is automatically satisfied in the setting of this paper)}.
Specifically, they proved that if the entropy-maximising fractional perfect matching $\xx^\star$ were not well-distributed, there would be a way to shift weight between high and low weight edges to increase the entropy, contradicting the entropy-maximality of $\xx^\star$. (This shifting is done with a 4-cycle, which can actually be interpreted as the $k=2$ case of the shifting structure defined in \cref{def:shifting-structure}).

One might hope to do the same for higher-uniformity hypergraphs,  using the shifting operation defined in the previous section. Unfortunately, the Cuckler--Kahn argument does not seem to generalise to higher-uniformity hypergraphs: we were not able to rule out the possibility that the maximum-entropy fractional matching $\xx^\star$ has a very high-weight edge $e$ (and is therefore not well-distributed), but that there is a ``conspiratorial'' arrangement of zero-weight edges 
in all the shifting structures which involve $e$, which prevent us from using shifting operations to increase the entropy of $\xx^\star$ by redistributing the weight of $e$.

To overcome this, we introduce an ``annealing'' technique, using the well-distributed fractional perfect matching $\hat \xx$ from \cref{lem:exist-normal}. Indeed, for the maximum-entropy fractional perfect matching $\xx^\star$, we can consider the ``annealed'' fractional perfect matching $\xx=(1-\delta)\xx^\star+\delta\hat \xx$, for some small $\delta>0$. The contribution from $\hat\xx$ means that all edges have weight at least $\Omega_{k,\gamma}(\delta/n^{k-1})$ 
, meaning that we can use \cref{lem:weight-shift,lem:find-structure} to perform shifting operations to redistribute some of the weight on high-weight edges, increasing the entropy. Either we can do so many shifting operations that the entropy rises above $h(\xx^\star)$ (i.e., the entropy gain from shifting overcomes the entropy loss from annealing), or we end up at a well-distributed fractional perfect matching (whose entropy is not much less than $h(\xx)$, if the annealing parameter $\delta$ is sufficiently small). The former case is impossible, by the entropy-maximality of $\xx^\star$, so we obtain the desired well-distributed fractional perfect matching.

\begin{proof}[Proof of \cref{lem:exist-normal-entropy}]
Throughout this proof, we can (and do) assume that $n$ is large, and $\varepsilon$ is small, with respect to $k,\gamma$.

By \cref{lem:exist-normal,lem:normal-entropy}, there is $C=O(1)$ and a fractional perfect matching $\hat{\xx}$ such that $h(\xx^\star)-h(\hat{\xx})\le C n$, and $\hat{\xx}[e]\ge 1/(Cn^{k-1})$ for each edge $e$.

Let $\xx_0=(1-\varepsilon/C) \xx^\star+ (\epsilon/C)\hat\xx$. We think of $\xx_0$ as an ``annealed'' version of $\xx^\star$, which has slightly lower entropy but has a guaranteed lower bound on the weight of each edge. Indeed, by concavity of the entropy function $x\mapsto x\ln(1/x)$, we have
\begin{equation}h(\xx_0)\geq (1-\varepsilon/C)h(\xx^\star) - (\varepsilon/C)h(\hat{\xx}) =  h(\xx^\star)-(\varepsilon/C)(h(\xx^\star)-h(\hat \xx))\ge h(\xx^\star)-\varepsilon n\,,\label{eq:near-maximal-entropy}\end{equation}
i.e., $\xx_0$ still has near-maximum entropy, and we have
\begin{equation}\xx_0[e]\ge \frac{\varepsilon}{C^2n^{k-1}}=\frac{\Omega_{k,\gamma}(\varepsilon)}{n^{k-1}}\text{ for each edge }e\,.\label{eq:weight-lower-bound}\end{equation}

Now, recall that our objective is to find a fractional perfect matching $\xx$ satisfying $h(\xx)\ge h(\xx^\star)-\varepsilon n$ and $\Omega_{k,\gamma}(\varepsilon^{3k})/n^{k-1}\le \xx[e]\le O(\varepsilon^{-3k})/n^{k-1}$ for every edge $e$. Note that $\xx_0$ already satisfies the desired entropy lower bound, and the desired lower bound on all the edge-weights; the only problem is that it might not satisfy the desired upper bound on the edge-weights. We will use a sequence of shift operations, starting from $\xx_0$, to obtain a fractional perfect matching which satisfies all the desired properties.

\medskip\noindent\textbf{An iterated shifting algorithm.} 
Define
\[\eta=\frac{4/\gamma}{\binom{n-1}{k-1}}=\frac{\Omega_{k,\gamma}(1)}{n^{k-1}}, \qquad \Delta=\frac{\varepsilon}{2C^2n^{k-1}}=\frac{O(\varepsilon)}{n^{k-1}}, \qquad D=\varepsilon^{-3k}.\] 
We greedily shift to obtain a sequence of fractional perfect matchings $\xx_1,\dots,\xx_T$, according to the following algorithm. For each $t\ge 1$, a \emph{good configuration} at step $t$ is a pair of edges $(e,f)$, intersecting in a single vertex, together with a shifting structure for $(e,f)$, such that (with respect to this shifting structure) we have $\xx_t[e_1]\ge D/n^{k-1}\ge 2\Delta$ (for small enough $\varepsilon$)
and $\xx_t[e_1],\dots,\xx_t[e_k]\ge 2\Delta$ and $\xx_t[f_1],\dots,\xx_t[f_k]\le \eta-\Delta$. If there is at least one good configuration at step $i$, then we choose one arbitrarily and set $\xx_{i+1}=\operatorname{shift}_\Delta(\xx_i,\mathcal U)$ (where $\mathcal U$ is the shifting structure associated with the good configuration). Otherwise, if no good configuration exists, set $T=i$ 
and terminate the algorithm.

Assuming $\varepsilon$ is sufficiently small, by \cref{lem:weight-shift} each step of our algorithm significantly increases the entropy. Namely, for each $0\le i<T$ we have
\[h(\xx_{i+1})\ge h(\xx_i)+\Delta\ln\left(\frac{D \Delta^{k-1}}{2n^{k-1}\eta^k}\right)\ge h(\xx_i)+\Delta\ln(D)/2\]
(here we are using that $\Delta^{k-1}/(2n^{k-1}\eta^k) \geq 1/\sqrt{D}$, 
assuming that $\varepsilon$ is sufficiently small with respect to $k,\gamma$).
So, recalling \cref{eq:near-maximal-entropy} and the definitions of the parameters $\Delta,D$, we have \begin{equation}T\le \frac{\varepsilon n}{\Delta\ln(D)/2}=O\biggl(\frac{n^k}{\ln(1/\varepsilon)}\biggr)\,.\label{eq:T-bound}\end{equation}

\medskip\noindent\textbf{Reducing to an edge-weight upper bound.} Recalling \cref{eq:near-maximal-entropy}, and noting that each shifting operation can only increase the entropy, we have 
\[h(\xx_T)\ge h(\xx_0)\ge h(\xx^\star)-\varepsilon n\,.\]
Also, recall from \cref{eq:weight-lower-bound} that in $\xx_0$ every edge has weight at least $2\Delta$. We only ever decrease the weight of an edge (by $\Delta$) when the weight of that edge is at least $2\Delta$, so in $\xx_T$ every edge has weight at least $\Delta\ge O(\varepsilon)/n^{k-1}$. That is to say, to prove that $\xx=\xx_T$ is  $O_{k,\gamma}(\varepsilon^{-3k})$-well-distributed (completing the proof of the lemma), it suffices to show that in $\xx_T$ every edge has weight less than $D/n^{k-1}=O(\varepsilon^{-3k})/n^{k-1}$.

\medskip\noindent\textbf{High-weight edges imply good configurations.}
Suppose for the purpose of contradiction that there is an edge $e$ with $\xx_T[e]\ge D/n^{k-1}$. We will show that there is a good configuration at step $T$ (contradicting the definition of $T$ as the final step of our iterated shifting algorithm). This will be accomplished by deleting a few edges of $G$ (whose weights are not compatible with a good configuration) and then applying \cref{lem:find-structure}.

First, let $G_{>\eta-\Delta}$ be the $k$-graph consisting of all edges $g$ of $G$ with $\xx_T[g]>\eta-\Delta$. By the definition of $\eta$ and $\Delta$, assuming $\varepsilon$ (therefore $\Delta$) is sufficiently small, we have $\eta-\Delta>0$ and $1/(\eta-\Delta) \le (\gamma/2)\binom{n-1}{k-1}$. 
For every vertex $v$, since $\xx_T$ is a fractional perfect matching, the sum of weights of edges including $v$ is exactly 1, so there are at most $1/(\eta-\Delta)\le (\gamma/2)\binom{n-1}{k-1}$ edges containing $v$ which have weight greater than $\eta-\Delta$. That is to say, $G_{>\eta-\Delta}$ has maximum degree at most $(\gamma/2)\binom{n-1}{k-1}$.

Then, let $G_{<2\Delta}$ be the $k$-graph consisting of the edges $g$ of $G$ which satisfy $\xx_T[g]<2\Delta$. Recalling \cref{eq:weight-lower-bound}, each edge of $G_{<2\Delta}$ must have been involved in one of our $T$ shifting operations (each of which decreases the weight of exactly $k$ edges), so by \cref{eq:T-bound}, $G_{<2\Delta}$ has at most $kT\le O(n^k/\ln(1/\varepsilon))\le (\gamma/(8k))\binom nk$ edges (assuming $\varepsilon$ is sufficiently small).
That is to say, $G_{<2\Delta}$ has average degree at most $(\gamma/(8k))\binom{n-1}{k-1}$, meaning that it has at most $n/(4k)$ vertices which have degree greater than $(\gamma/2)\binom{n-1}{k-1}$. Call these vertices \emph{bad}.

Fix a vertex $v_1\in e$. By \cref{cla:monotone}, $v_1$ is contained in at least $(1/2+\gamma)\binom{n-1}{k-1}$ edges of $G$. At most $(n/(4k)+k-1)\binom{n-2}{k-2}\le (1/2)\binom{n-1}{k-1}$ of these edges also contain a bad vertex or a vertex in $e\setminus\{v_1\}$. So, there is an edge $f$, containing no bad vertices, such that $e\cap f=\{v_1\}$.

Let $G'$ be the hypergraph obtained from $G$ by deleting all edges of $G_{>\eta-\Delta}$ which do not intersect $f$, and deleting all edges of $G_{<2\Delta}$ which do not intersect $e$. By construction, we have deleted at most $(\gamma/2)\binom{n-1}{k-1}$ edges incident to every vertex in $(e\cup f)\setminus\{v_1\}$. By \cref{cla:monotone,fact:dirac-bound}, we have $\delta_1(G)\ge (1/2+\gamma)\binom{n-1}{k-1}$, so in $G'$, every vertex in $(e\cup f)\setminus\{v_1\}$ has degree at least $(1/2+\gamma/2)\binom{n-1}{k-1}$.  Using \cref{lem:find-structure}, we deduce that there is a shifting structure $\mathcal U$ for $(e,f)$ in $G'$. By the definition of $G'$ via $G_{>\eta-\Delta}$ and $G_{<2\Delta}$, with respect to $\mathcal U$ we have $\xx_2[e_1],\dots,\xx_2[e_k]\ge 2\Delta$ and $\xx_2[f_1],\dots,\xx_2[f_k]\le \eta-\Delta$. Recalling our assumption $\xx_T[e_1]=\xx_T[e]\ge D/n^{k-1}$, we see that $\mathcal U$, together with $(e,f)$, forms a good configuration, yielding the desired contradiction.
\end{proof}

\section{A random greedy matching process} \label{section:random-greedy}
The goal of this section is prove the lower bound in \Cref{thm:num-of-PM}. We do this by analysing a random greedy matching process, guided by a fractional perfect matching.

\begin{definition}[Random greedy matching process]\label{def:random-greedy}
Consider a $k$-graph $G$, and let $\xx$ be a fractional perfect matching in $G$. 
Then, the \emph{greedy matching process} in $G$ with respect to $\xx$ is defined
as follows. Let $G(0)=G$ and for each $i\ge1$: choose a random edge
$e(i)$ in $G(i-1)$, with probability proportional to $\xx[e(i)]$,
and delete all the vertices in $e(i)$ from $G(i-1)$ to obtain a
$k$-graph $G(i)$. This process cannot continue forever: let $G(M)$ be the first graph
in the sequence with no positive-weight edges.
\end{definition}

First, in \cref{sec:random-greedy-analysis}, we consider certain statistics that evolve with the process, and show that these statistics concentrate around certain explicit values. This analysis requires that $\xx$ is well-distributed (recall the definition from \cref{def:well-distributed}).
Then, in \cref{sec:random-greedy-deduction}, we use this analysis to show that the greedy matching process gives rise to a high-entropy probability distribution on perfect matchings. By \cref{cla:unif-maximizes}, this implies a lower bound on the number of perfect matchings.

\subsection{Analysis of the greedy matching process}\label{sec:random-greedy-analysis}
Our objective in this section is to show that if $\xx$ is $n^{c}$-well-distributed for sufficiently small $c$, then our greedy matching process is likely to continue until nearly all vertices are gone. We are also able to track the evolution of various important statistics of the process. These statistics are chosen to evince that the almost-perfect matching produced by the process can typically be completed to a perfect matching, and that the resulting distribution on perfect matchings has high entropy.
\begin{theorem}
\label{lem:greedy-matching-analysis}For any constant $k\ge 2$, there is a constant $c>0$ such
that the following holds. Let $G$ be a $k$-graph, and let $\xx$ be an $n^{c}$-well-distributed fractional perfect matching of $G$. Then, consider
the greedy matching process defined in \cref{def:random-greedy} (in $G$, with respect to $\xx$), producing a sequence of $k$-graphs $G(0),\dots,G(M)$. With probability at least $1-\exp(-{n^c})$: we have 
\begin{equation}
M\ge(1-n^{-c})n/k\,,\label{eq:M-1}
\end{equation}
and for each $i\le (1-n^{-c})n/k$ and each set $S$ of at most $k-1$ vertices of $G(i)$,
\begin{align}
\sum_{e\in E(G(i))}\xx[e] & =(1\pm n^{-c})\left(\frac{n/k-i}{n/k}\right)^{k}\frac{n}{k}\,,\label{eq:weight-2}\\
\sum_{e\in E(G(i))}\xx[e]\ln\left(\frac{1}{\xx[e]}\right) & =(1\pm n^{-c})\left(\frac{n/k-i}{n/k}\right)^{k}h(\xx)\,,\label{eq:entropy-3}\\
\deg_{G(i)}(S)&=\left(\frac{n/k-i}{n/k}\right)^{k-|S|}\deg_{G}(S)\pm n^{k-|S|-c}\,.\label{eq:degree-4}
\end{align}
\end{theorem}

To prove \cref{lem:greedy-matching-analysis}, we will use an approach popularised by Bohman, Frieze and Lubetzky~\cite{BFL10,BFL15}: we anticipate the trajectories of the relevant statistics, and use the difference between the actual and anticipated trajectories, to define supermartingales to which we can apply an exponential tail bound for martingales, namely \emph{Freedman's inequality}\footnote{Freedman's inequality was originally stated for martingales, but it also holds for
supermartingales with the same proof.}~\cite[Theorem~1.6]{Fre75}. The statement of Freedman's inequality is as follows, writing $\Delta X(i)$ for the one-step change $X(i+1)-X(i)$.
\begin{lemma}
\label{lem:freedman}Let $X(0),X(1),\dots$ be a supermartingale\footnote{Recall that a process $(X(i))_i$ is a \emph{supermartingale} with respect to a filtration $(\mathcal F_i)_i$ if $\mathbb E[\Delta X(i)\,|\,\mathcal F_i]\le 0$ for all $i$.} with
respect to a filtration $(\mathcal{F}_{i})_{i}$. Suppose that $|\Delta X(i)|\le K$
for all $i$, and let $V(i)=\sum_{j=0}^{i-1}\mathbb{E}\left[(\Delta X(j))^{2}\,\middle|\,\mathcal{F}_{j}\right]$.
Then for any $t,v>0$,
\[
\Pr\left[X(i)\ge X(0)+t\mbox{ and }V(i)\le v\mbox{ for some }i\right]\le\exp\left(-\frac{t^{2}}{2(v+Kt)}\right)\,.
\]
\end{lemma}

We remark that, for the graph case ($k=2$), Cuckler and Kahn~\cite{CK-counting} also analysed a random greedy process, though their analysis was much more complicated. Partially this is due to their consideration of Hamilton cycles as well as perfect matchings, but it is also due to the fact that the field of random (hyper)graph processes has become much more mature (with more streamlined methods available) since \cite{CK-counting} first appeared.

\begin{proof}[Proof of \cref{lem:greedy-matching-analysis}]
Instead of specifying the value of $c$ at the outset, we will assume
that $c$ is sufficiently small to satisfy certain inequalities that
arise in few places throughout the proof.

It is convenient to ``freeze'' the process instead of aborting it at the point where no positive-weight edges remain: recalling that $G(M)$ is the first graph in the sequence with no positive-weight edges, we artificially define $G(i)=G(M)$ for $i\ge M$. 

\medskip\noindent\textbf{Tracked statistics.} First, we define an ensemble of statistics
that we will keep track of as the process evolves.

For a set $S$ of at most $k-1$ vertices, we write $\Gamma_{S}(i)$ for the collection of
sets $U$ of $k-|S|$ vertices such that $S\cup U$ is an edge
of $G(i)$. For each vertex $v$ we define a process $f_{v}$ by
\[
f_{v}(i)=\sum_{U\in\Gamma_{v}(i)}\xx[\{v\}\cup U]\,.
\]
Also, for each set $S$ of at most $k-1$ vertices of $G$, such that $\deg_G(S)\ge n^{k-|S|-c}$, define a process
\[
d_{S}(i)=|\Gamma_S(i)|\,.
\]
Finally, define a process $a$ by
\[
a(i)=\sum_{e\in E(G(i))}\xx[e]\ln\left(\frac{1}{\xx[e]}\right).
\]
Note that $E(G(i))=\Gamma_{\emptyset}(G(i))$ (so we can interpret $a(i)$ as a sum over elements of $\Gamma_{\emptyset}(G(i))$).

Note that \begin{equation}f_{v}(0)=1\,,\qquad d_S(0)=\deg_G(S)\ge n^{k-|S|-c}\,,\qquad a(0)=h(\xx)=\Theta(n\ln n)\label{eq:q0}\end{equation} for each appropriate $v,S$ (for the final estimate we used \cref{cla:jensen-entropy}).

We organise all these processes $f_v,d_S,a$ into collections $\mathcal Q_1,\dots,\mathcal Q_{k}$, as follows. Let $\mathcal Q_{r}$ contain the processes of the form $d_S$ with $|S|=k-r$, and in addition put all $n$ processes of the form $f_{v}$ in $\mathcal Q_{k-1}$ and put $a$ in $\mathcal Q_k$. The idea is that every process in $\mathcal Q_r$ can be interpreted as a sum over $r$-sets in $\Gamma_S(i)$, for some $S$ of size $k-r$.

Then, let $\mathcal Q=\mathcal Q_1\cup \dots \cup \mathcal Q_k$. 

\medskip\noindent\textbf{Anticipated trajectories.} Let 
\[
p(i)=\frac{n-ki}{n}= \frac{n/k-i}{n/k} = 1 - \frac{ki}{n}\,.
\]
This is the proportion of vertices that remain in $G(i)$ (unless
the process freezes before then). For $q\in \mathcal Q_r$, we anticipate that $q(i)$ will evolve like $p(i)^rq(0)$,
because it is a sum over $r$-sets of vertices.

\medskip\noindent\textbf{Error tolerances.} 
Let $C=1/(8c)$ and let $\varepsilon(i)=p(i)^{-C}n^{-1/4}$. We wish to show that,
with high probability, for all $1\le r\le k$ 
and $q\in\mathcal{Q}_r$ and all $i\le(1-n^{-c})n/k$,
\[
(1-\varepsilon(i))p(i)^rq(0)\le q(i)\le(1+\varepsilon(i))p(i)^rq(0)\,.
\]
Let $T_{0}$ be the smallest index $i$ such that one of the these
inequalities is violated (let $T_{0}=\infty$ if this never happens). Writing $x\land y=\min(x,y)$, let $T=T_{0}\land(1-n^{-c})n/k$. For $r\le k$ and $q\in\mathcal{Q}_r$, define processes $F^-q,F^+q$ by
\begin{align*}
F^+q(i) & =q(i\land T)/q(0)-(1+\varepsilon(i\land T))p(i)^r,\\
F^-q(i) & =-q(i\land T)/q(0)+(1-\varepsilon(i\land T))p(i)^r.
\end{align*}
Note that
\begin{equation}
    F^+q(0)=F^-q(0)=-n^{-1/4}.\label{eq:starting-diff}
\end{equation}

\medskip\noindent\textbf{The plan.} We want to show that for each $q\in\mathcal{Q}$
and each $s\in\left\{ +,-\right\} $, the process $F^sq$ is a supermartingale,
and then we want to use \cref{lem:freedman} and the union bound to
show that whp each $F^sq$ only takes negative values, i.e., $T_0>(1-n^{-c})n/k$. 

Before doing this, we show that the statement of \cref{lem:greedy-matching-analysis}
follows. Indeed, assume that $T_0>(1-n^{-c})n/k$; we will deduce \cref{eq:M-1,eq:weight-2,eq:entropy-3,eq:degree-4}. 

So, consider $i\le (1-n^{-c})n/k<T_0$. Note that $p(i)>n^{-c}$, so $\varepsilon(i)<n^{-1/8}<n^{-c}$.
For each vertex $v$ and $s\in \{-,+\}$, since $F^sf_{v}(i)<0$ we have $f_v(i)=(1\pm\varepsilon(i))p(i)^{k-1}$, and
\begin{equation}
\sum_{e\in E(G(i))}\xx[e]=\frac{1}{k}\sum_{v\in V(G(i))}f_{v}(i)=\frac{1}{k}\cdot p(i)n\cdot (1\pm\varepsilon(i))p(i)^{k-1}\,,
\label{eq:weight-intermediate}
\end{equation}
which implies \cref{eq:weight-2}.

In particular, since the above expression is positive, whenever $i\le (1-n^{-c})n/k$
there is at least one positive-weight edge in $G(i)$. 
That is to say, the process
cannot freeze until after step $(1-n^{-c})n/k$, i.e., $(1-n^{-c})n/k=T\le M$, which
is \cref{eq:M-1}. For \cref{eq:entropy-3}, 
again consider $i\le T=(1-n^{-c})n/k<T_0$.
Since $F^s a<0$ for each $s\in\{-,+\}$,
recalling \cref{eq:q0} we have
\begin{align*}
\sum_{e\in E(G(i))}\xx[e]\ln\left(\frac{1}{\xx[e]}\right) =a(i)=(1\pm n^{-1/8})\cdot p(i)^{k} h(\xx)
\end{align*}
which implies \cref{eq:weight-2}. Finally, note that \cref{eq:degree-4} trivially holds when $\deg_G(S)\le n^{k-|S|-c}$. For any other $S\subseteq V(G(i))$, using that $F^sd_S(i)<0$, we see that 
\[\deg_{G(i)}(S)=|\Gamma_S(i)|=(1\pm n^{-1/8})p(i)^{k-|S|}\deg_{G(i)}(S)\,,\]
which implies \cref{eq:degree-4}.

\medskip\noindent\textbf{Expected and worst-case differences.} We now estimate the
expected and worst-case per-step differences $\Delta f_{v}(i),\Delta a(i),\Delta d_S(i)$. (Recall that $\Delta X(i)$ means $X(i+1)-X(i)$). 

Fix a set $S$ of at most $k-1$ vertices, and consider $U\in\Gamma_{S}(i)$, for $i<T$. The only
way we can have $U\notin\Gamma_{S}(i+1)$ is if $e(i)$ intersects
$U$. For each vertex $u\in U$, the total weight of edges in $G(i)$
which intersect $u$ is $f_{u}(i)=(1\pm\varepsilon(i))p(i)^{k-1}$ (since we are assuming $i<T\le T_0$). 
The total weight of the edges which intersect more than one vertex in
$U$ is at most $\binom{k}{2}\binom{n}{k-2}(n^{c}/n^{k-1})=O(n^{-1+c})$ 
, recalling
that $\xx$ is $n^{c}$-well-distributed.

So, the total weight of edges intersecting $U$ in $G(i)$ is
\[
(1\pm\varepsilon(i))(k-|S|)p(i)^{k-1}-O(n^{c-1})=(1\pm O(\varepsilon(i)))(k-|S|)p(i)^{k-1}\,.
\]
On the other hand, \cref{eq:weight-intermediate} gives an expression
for the total weight of all edges in $G(i)$. Dividing these two quantities
by each other, we see that 
\[
\Pr\left[U\notin\Gamma_{S}(i+1)\,\middle|\, G(1),\dots,G(i)\right]=(1\pm O(\varepsilon(i)))\frac{k-|S|}{p(i)\cdot n/k}\,.
\]
Now, for $1\le r\le k$ and $q\in \mathcal Q_r$, 
we have $q(i)=(1\pm \varepsilon(i))p(i)^rq(0)$ (still assuming that $i<T$). Recall that $q(i)$ can be interpreted as a sum over elements of $\Gamma_S(i)$, for some $S$ of size $k-r$, so by linearity of expectation,
\begin{align}
\mathbb E\left[\Delta q(i)\,\middle|\, G(1),\dots,G(i)\right] & =-(1\pm O(\varepsilon(i)))\frac{rp(i)^{r-1}}{n/k}q(0)\nonumber \\
&=-(1\pm O(\varepsilon(i)))p(i)^rq(0)\cdot \frac{r}{p(i)\cdot n/k}\nonumber\\
 & =-\frac{rp(i)^{r-1}}{n/k}q(0)\pm O\left(\frac{\varepsilon(i)p(i)^{r-1}}{n}q(0)\right)\,.\label{eq:EDeltaf}
\end{align}
Note also that we have the deterministic bound 
\begin{equation*}
|\Delta f_{v}(i)|\le k\binom{n}{k-2}(n^{c}/n^{k-1})=O(n^{c-1})\,,
\end{equation*}
using the $n^{c}$-well-distributedness of $\xx$, and the fact that deleting the
$k$ vertices of an edge results in at most $k\binom{n}{k-2}$ edges
of $\Gamma_{v}(i)$ being deleted. Similarly, we have 
\begin{equation*}
|\Delta a(i)|\le k\binom{n}{k-1}\left((n^{c}/n^{k-1})\ln\left(\frac{1}{n^{-c}/n^{k-1}}\right)\right)=O(n^{c}\ln n)\,.
\end{equation*}
and
\begin{equation*}
|\Delta d_S(i)|\le k\binom{n}{k-|S|-1}=O(n^{k-|S|-1})\,.
\end{equation*}
Recalling \cref{eq:q0}, for all $q\in \mathcal Q$ we deduce
\begin{equation}|\Delta q(i)|\le O(n^{c-1})\cdot q(0)\,.\label{eq:Deltaq}\end{equation}

\medskip\noindent\textbf{Computations on the trajectories and tolerances.} In order
to estimate expressions of the form $\Delta F^sq$ (for $q\in\mathcal{Q}$
and $s\in\{-,+\}$), we also need some control over expressions of
the form $\Delta p^r$ and $\Delta(\varepsilon p^r)$ (where we write $p^{r}$ for the function $i\mapsto p(i)^{r}$ and we write $\varepsilon p^{r}$ for the function $i\mapsto\varepsilon(i)p(i)^{r}$). The
required computations are basically a calculus exercise.

For any $r=O(1)$, we have
\begin{align}
\Delta p^{r}\left(i\right) & =\left(1-\frac{i+1}{n/k}\right)^{r}-\left(1-\frac{i}{n/k}\right)^{r}\nonumber \\
 & =\left(1-\frac{i}{n/k}\right)^{r}\left(\left(\frac{n/k-i-1}{n/k-i}\right)^{r}-1\right)\nonumber \\
 & =p^{r}(i)\left(\left(1-\frac{1}{n/k-i}\right)^{r}-1\right)\nonumber \\
 & =p^{r}(i)\left(-\frac{r}{n/k-i}+O\left(\frac{1}{(n/k-i)^{2}}\right)\right)\nonumber \\
 & =-\frac{rp^{r-1}(i)}{n/k}\left(1+O\left(\frac{p(i)}{n}\right)\right)\nonumber \\
 & =-\frac{rp^{r-1}\left(i\right)}{n/k}+o\left(\frac{\varepsilon(i)p^{r-1}(i)}{n}\right)\,,\label{eq:Deltapr}
\end{align}
where we used the fact that $p(i)/\varepsilon(i) = o(n)$ in the last inequality (with lots of room to spare) and we deduce
\begin{align}
\Delta(\varepsilon p^{r})(i) & =n^{-1/4}\Delta p^{r-C}(i)\nonumber \\
 & \ge n^{-1/4}\frac{(C-r-o(1))p^{r-C-1}\left(i\right)}{n/k}\nonumber \\
 & =\frac{(C-r-o(1))\varepsilon\left(i\right)p^{r-1}\left(i\right)}{n/k}\,.\label{eq:Deltaepr}
\end{align}

\medskip\noindent\textbf{Verifying the supermartingale property.} We are assuming $c$
is small (so $C=1/(8c)$ is large). So, for any $r\le k$, $q\in \mathcal Q_r$ and
$i<T$, combining \cref{eq:EDeltaf,eq:Deltapr,eq:Deltaepr} we see
that
\begin{align*}
\mathbb E\left[\Delta F^+q(i)\,\middle|\, G(1),\dots,G(i)\right] & =\mathbb E\left[\Delta q(i)\,\middle|\, G(i)\right]/q(0)-\Delta p^{r}(i)-\Delta(\varepsilon p^{r})(i)\le 0\,.
\end{align*}
 and similarly
\[
\mathbb E\left[\Delta F^-q(i)\,\middle|\, G(1),\dots,G(i)\right]=-\mathbb E\left[\Delta q(i)\,\middle|\, G(i)\right]/q(0)+\Delta p^{r}(i)-\Delta(\varepsilon p^{r})(i)\le0\,.
\]
We have proved that each $F^sq$ is a supermartingale.

With a similar calculation, we see
\begin{equation}
    \mathbb E\left[\vphantom\sum|\Delta F^sq(i)|\,\middle|\,G(1),\dots,G(i)\right]\le O(1/n)\,.\label{eq:EabsDeltaq}
\end{equation}
for all $q\in \mathcal Q$ and $s\in\{-,+\}$.

\medskip\noindent\textbf{Applying Freedman's inequality.} Combining \cref{eq:Deltaq}, \cref{eq:Deltapr,eq:Deltaepr},
we see that (with
probability 1), for each $q\in \mathcal Q$ and $s\in\{-,+\}$,
\[
|\Delta F^sq|\le O(n^{c-1}).
\]
So, using \cref{eq:EabsDeltaq},
\[
\mathbb E\left[(\Delta F^sq(i))^{2}\,\middle|\, G(1),\dots,G(i)\right]\le O(n^{c-1})\cdot\mathbb E\left[\vphantom\sum|\Delta F^sq(i)|\,\middle|\,G(1),\dots,G(i)\right]\le O(n^{c-2})\,.
\]
Since $T\le n/k$ (and $\Delta F^sq(i)=0$ for $i\ge T$) this
yields
\[
\sum_{i=0}^{\infty}\mathbb E\left[(\Delta F^sq(i))^{2}\,\middle|\,G(1),\dots,G(i)\right]=O(n^{c-1})\,.
\]
Applying \cref{lem:freedman} with $t=-F^sq(0)=n^{-1/4}$ 
(recalling \cref{eq:starting-diff}) and $v=O(n^{c-1})$
then gives
\[
\Pr\left[F^sq(i)\ge 0\mbox{ for some }i\right]
\le\exp\left(-\Omega\left(\frac{(n^{-1/4})^{2}}{n^{c-1}+n^{c-1}\cdot n^{-1/4}}\right)\right)
=\exp(-n^{1/4})\,.
\]
The desired result then follows from a union bound over $q\in\mathcal{Q}$
-and $s\in\{-,+\}$.
\end{proof}

\subsection{Deducing the lower bound in \texorpdfstring{\cref{thm:num-of-PM}}{Theorem~\ref{thm:num-of-PM}}}\label{sec:random-greedy-deduction}
Now we prove the lower bound in \cref{thm:num-of-PM}, using \cref{lem:greedy-matching-analysis,lem:exist-normal-entropy}.

\begin{proof}[Proof of the lower bound in \cref{thm:num-of-PM}]
    Let $\mathcal{S}(G)$ be the set of \emph{sequences }of disjoint edges
$(e(1),\dots, \allowbreak e(n/k))$ (i.e, perfect matchings with an ordering on
their edges). Our objective is to prove that 
\begin{equation}
\ln\left(|\mathcal{S}(G)|\right)=h(G)+\frac{n}{k}\ln\left(\frac{n}{k}\right)-n + o(n)\,,\label{eq:hope-entropy}
\end{equation}
Recalling that $\Phi(G)$ denotes the number of perfect matchings in $G$, we have $|\calS(G)|=(n/k)!\Phi(G)$. So, given \cref{eq:hope-entropy}, the desired result on $\Phi(G)$ will follow, using Stirling's approximation $\ln((n/k)!)=(n/k)\ln(n/(ek))+o(n)$.

\medskip\noindent\textbf{The plan.} Let $c$ be as in \cref{lem:greedy-matching-analysis}. By \cref{lem:exist-normal-entropy}, there is an $n^{c}$-well-distributed
fractional perfect matching $\xx$ with 
\[
h(\xx)\ge h(G)-O(n^{1-c/(3k)})=h(G)-o(n)\,.
\]

We will consider the random greedy matching process in \cref{def:random-greedy}, guided
by the fractional perfect matching $\xx$. By \cref{lem:greedy-matching-analysis}, this almost always
produces a sequence of disjoint edges $e(1),\dots,e(I)$, for some
$I=(1-o(1))n/k$, which can be completed to an ordered perfect matching
$(e(1),\dots,e(n/k))$ (by \cref{eq:degree-4}, and the fact that $G$ is a $(d,\gamma)$-Dirac
$k$-graph). That is to say, the random greedy matching process gives
rise to a probability distribution on $\mathcal{S}(G)$, and the entropy
of this probability distribution gives a lower bound on $\ln\left(|\mathcal{S}(G)|\right)$
(recalling \cref{cla:unif-maximizes}). We will estimate the desired entropy using the chain
rule (\cref{cla:chain-rule}), and \cref{eq:entropy-3} and \cref{eq:weight-2}. (The details of this are a bit fiddly, as
we need to handle the small probability that the random greedy matching process fails to satisfy the conclusion of \cref{lem:greedy-matching-analysis}). 

\medskip\noindent\textbf{Random greedy matching.} Let $I=(1-n^{-c/(2k)})n/k$. We run
the random greedy matching process defined in \cref{def:random-greedy} for $I$ steps to
produce a sequence of disjoint edges $e(1),\dots,e(M\land I)$, where
$M$ is the first step for which $G(M)$ has no positive-weight edges (here we write $x\land y=\min(x,y)$).
If $M<I$, we artificially set $e(i)=*$ for $M\land I<i\le I$.

Let $p(i)=(n/k-i)/(n/k)$ and let $\mathcal{E}$ be the event that the
estimates in \cref{eq:M-1,eq:weight-2,eq:entropy-3,eq:degree-4} (in \cref{lem:greedy-matching-analysis}) hold for the first $I$ steps. In particular, this
implies that $M\ge I$, and it implies that for $i\le I$,
\begin{align}
\sum_{e\in E(G(i))}\xx[e] & =(1\pm n^{-c})p(i)^{k}\frac{n}{k}\,,\label{eq:weight-use-1}\\
\sum_{e\in E(G(i))}\xx[e]\ln\left(\frac{1}{\xx[e]}\right) & \ge(1\pm n^{-c})p(i)^{k}h(\xx)\ge p(i)^{k}(h(G)-o(n))\,,\label{eq:entropy-use-2}\\
\delta_{d}(G(i)) & \ge p(i)^{k-d}\delta_{d}(G)-n^{k-|S|-c}\ge(\alpha_{d}(k)+\gamma/2)\binom{p(i)n-d}{k-d}\,.\nonumber 
\end{align}
When $\mathcal{E}$ occurs, the definition of $\alpha_{d}(k)$ implies
that $G(i)$ has a perfect matching, meaning that we can extend $e(1),\dots,e(I)$
to a sequence of $n/k$ disjoint edges $e(1),\dots,e(n/k)$. Using \cref{cla:unif-maximizes}, this implies that
\[
\ln\left(|\calS(G)|\right)\ge H\left(e(1),\dots,e(I)\,\middle|\,\mathcal{E}\right)\,.
\]

\medskip\noindent\textbf{Removing the conditioning.} Now, let $I_{\mathcal{E}}$ be
the indicator random variable for the event $\mathcal{E}$, and let
$\mathcal{E}^{c}$ be the complement of $\mathcal{E}$. 
Observe that $H(I_{\mathcal E}\mid e(1),\dots, e(I)) = 0$ since the outcome of $e(1),\dots, e(I)$ completely determines $I_{\mathcal E}$. 
Thus, on the one hand, by \cref{cla:chain-rule}, we have 
$$H(e(1),\dots,e(I),I_{\mathcal E}) = H(e(1),\dots, e(I)) + H(I_{\mathcal E}\mid e(1),\dots, e(I)) = H(e(1),\dots, e(I))\,.$$ 
On the other hand, using \cref{cla:chain-rule} again, we have
\begin{align*}
H(e(1),\dots,e(I), I_{\mathcal E}) 
&=H(I_{\mathcal{E}})+H\left(e(1),\dots,e(I)\,\middle|\, I_{\mathcal{E}}\right)\\
 & =H(I_{\mathcal{E}})+\Pr[\mathcal{E}]\cdot H\left(e(1),\dots,e(I)\,\middle|\,\mathcal{E}\right)+\Pr[\mathcal{E}^{c}]\cdot H\left(e(1),\dots,e(I)\,\middle|\,\mathcal{E}^{c}\right).
\end{align*}
By \cref{lem:greedy-matching-analysis}, we have $\Pr[\mathcal{E}]=1-n^{-\omega(1)}$. Also note
that $H\left(e(1),\dots,e(I)\,\middle|\,\mathcal{E}^{c}\right)\le\ln\left(\binom{n}{k}!\right)=O(n^{k}\ln n)$.
So, we deduce
\begin{equation}
H\left(e(1),\dots,e(I)\,\middle|\,\mathcal{E}\right)\ge H(e(1),\dots,e(I))-o(n)\,.\label{eq:entropy-conditional}
\end{equation}
So, it suffices to study $H(e(1),\dots,e(I))$. 

\medskip\noindent\textbf{Computing the entropy.} By \cref{cla:chain-rule} we have
\begin{equation}
H(e(1),\dots,e(I))=\sum_{i=0}^{I-1}H\left(e(i+1)\,\middle|\, e(1),\dots,e(i)\right)\,.\label{eq:chain-rule-use}
\end{equation}
Consider any $i\le I$, and consider some possible outcomes $e_{1},\dots,e_{i}$
of $e(1),\dots,e(i)$. By \cref{cla:unif-maximizes}, we always have the crude bound
\[
H\left(e(i+1)\,\middle|\, e(1)=e_{1},\dots,e(i)=e_{i}\right)\le\ln\left(\binom{n}{k}\right)=O(\ln n)\,.\label{eq:crude-entropy}
\]
On the other hand, suppose that $e_{1},\dots,e_{i}\ne*$. Then these
outcomes of $e(1),\dots,e(i)$ determine an outcome $G_{i}$ of $G(i)$;
let $\mu_{i}=\sum_{e\in E(G_{i})}\xx(e)\le(1+o(1))p(i)^{k}(n/k)$
(here we are using~\eqref{eq:weight-2}) and note that if $\mu_{i}>0$ then (by the definition of the greedy matching process)
\[
H\left(e(i+1)\,\middle|\, e(1)=e_{1},\;\dots,\;e(i)=e_{i}\right)=\sum_{e\in E(G_{i})}\frac{\xx[e]}{\mu_{i}}\ln\left(\frac{\mu_{i}}{\xx[e]}\right)=\frac{1}{\mu_{i}}\sum_{e\in E(G_{i})}\xx[e]\ln\left(\frac{1}{\xx[e]}\right)+\ln\mu_{i}\,.
\]
If $G(i)=G_{i}$ satisfies \cref{eq:weight-use-1,eq:entropy-use-2},
then we compute
\[
H\left(e(i+1)\,\middle|\, e(1)=e_{1},\;\dots,\;e(i)=e_{i}\right)=\frac{h(G)}{n/k}+\ln\left(\frac{n}{k}\right)+k\ln\left(\frac{n/k-i}{n/k}\right)-o(1)\,.
\]
Combining this with \cref{eq:crude-entropy}, and recalling that $\Pr[\mathcal E^c]=n^{-\omega(1)}$ (by \cref{lem:greedy-matching-analysis}), we have
\begin{align*}
H\left(e(i+1)\,\middle|\, e(1),\dots,e(i)\right) & \ge\Pr[\mathcal{E}]\left(\frac{h(G)}{n/k}+\ln\left(\frac{n}{k}\right)+k\ln\left(\frac{n/k-i}{n/k}\right)-o(1)\right)+\Pr[\mathcal{E}^{c}]\cdot O(\ln n)\\
 & \ge\frac{h(G)}{n/k}+\ln\left(\frac{n}{k}\right)+k\ln\left(\frac{n/k-i}{n/k}\right)-o(1)\,.
\end{align*}
Now, recall \cref{eq:chain-rule-use}. Recalling that $I=(1-n^{-c/(2k)})n/k$,
and computing 
\[
\sum_{i=1}^{I}k\ln\left(\frac{n/k-i}{n/k}\right)\ge k\ln\left(\frac{(n/k)!}{(n/k)^{n/k}}\right)= -n+o(n)
\]
using Stirling's approximation, we have
\[
H(e(1),\dots,e(I))=h(G)+\frac{n}{k}\ln\left(\frac{n}{k}\right)-n+o(n)\,.
\]
Via \cref{eq:chain-rule-use}, this implies the desired bound \cref{eq:hope-entropy}.
\end{proof}

\section{Upper bound in 
\texorpdfstring{\cref{thm:num-of-PM}}{Theorem~\ref{thm:num-of-PM}}
}\label{section:entropy-ub}
In this section we prove the upper bound in \cref{thm:num-of-PM}, as a consequence of the following technical theorem of Kahn~\cite[Theorem~4.2]{Kahn-Shamir}
(which is essentially a hypergraph generalisation of the main result of \cite{CK-entropy}).
We introduce some notation that will be used throughout the section:
fix $k\ge2$, let $G$ be a $k$-graph and let $f$
be a perfect matching in $G$. 
We implicitly regard $f$ as a set of edges.
\begin{itemize}
\item Let $\Lambda=(k-1)n/k$.
\item For a vertex $v$, let $f_{v}$ be the edge of $f$ containing $v$,
and let $f(v)= f_{v}\setminus\{v\}$.

\item For a vertex $v$, a perfect matching $f$, and a $(k-1)$-set of
vertices $Y$, let 
\[
T(v,f,Y)=\left\{ B\in f:B\ne f_{v},\,B\cap Y\ne\emptyset\right\} 
\]
and $\tau(v, f,Y)=|T(v,f,Y)|$.
\item For a vertex $v$ and a $(k-1)$-set of vertices $Y$, let $\Gamma_{v}(Y)$
be the set of perfect matchings $f$ satisfying $\tau(v,f,Y)<k-1$.
\end{itemize}
Also, given a (not necessarily uniform) random perfect matching $\mathbf f$ in $G$: for a vertex $v$ and a $(k-1)$-set of vertices $Y$, define $p_{v}(Y)=\Pr[\mathbf f(v)=Y]$ and $\gamma_{v}(Y)=\Pr[\mathbf{f}\in\Gamma_{v}(Y)]$.

\begin{theorem}\label{thm:technical-kahn}
Fix $k\ge2$, let $G$ be a $k$-graph and let $\mathbf{f}$
be a (not necessarily uniformly) random perfect matching in $G$.
Then 
\[
H(\mathbf{f})<\frac{1}{k}\sum_{v\in V(G)}H(\mathbf{f}(v))-\Lambda+O\left(\sum_{v\in V(G)}\sum_{Y}p_{v}(Y)\gamma_{v}(Y)^{1/(k-1)}\right)+O(\ln n)\,,
\]
where, in the third sum, $Y$ ranges over all $(k-1)$-subsets of $V(G)$.
\end{theorem}
In order to deduce the desired bound from \cref{thm:technical-kahn} we use the following auxiliary lemma, appearing as \cite[Lemma~4.3]{Kahn-Shamir}.
\begin{lemma}\label{lem:technical-kahn}
Consider $l,D,t\in \mathbb N$ and $\lambda,\varepsilon > 0$, and $p_{i},\gamma_{i}\in[0,1]$ for $i\in\{1,\dots,l\}$,
satisfying
\begin{itemize}
\item $l=nD,$
\item $\sum_{i=1}^l p_{i}=n$,
\item $n\ln(D)-\sum_{i=1}^l p_{i}\ln(1/p_{i})<\lambda n$
\item $\sum_{i=1}^l \gamma_{i}\le \varepsilon nD$.
\end{itemize}
Then 
\[
\sum_{i=1}^l p_{i}\gamma_{i}^{1/(t-1)}=\delta_{k,\lambda}(\varepsilon)n,
\]
for some $\delta_{k,\lambda}(\varepsilon)>0$ that tends to zero as $\varepsilon\to 0$ (holding $k,\lambda$ fixed).
\end{lemma}
In other words, if we can prove a $O(n)$ bound on the expression in the third bullet point, and an $o(nD)$ bound on the expression in the last bullet point, we deduce a $o(n)$ bound in the final expression.

We now deduce the upper bound in \cref{thm:num-of-PM}.
\begin{proof}
[Proof of the upper bound in \cref{thm:num-of-PM}]Let $\mathbf{f}$ be a uniformly random perfect matching
of $G$, and define a fractional matching $\xx$ by $\xx[e]=\Pr[e\in\mathbf{f}]$.
Then, for each vertex $v$ we have
\[
H(\mathbf{f}(v))=\sum_{e\ni v}\xx[e]\ln\left(\frac{1}{\xx[e]}\right)\,,
\]
so
\[
\sum_{v\in V(G)}H(\mathbf{f}(v))=kh(\xx)\,.
\]
The desired result will follow from \cref{thm:technical-kahn} if we can show that
\begin{equation}
\sum_{v\in V(G)}\sum_{S}p_{v}(S)\gamma_{v}(S)^{1/(k-1)}=o(n)\,,\label{eq:technical-hope}
\end{equation}
where, in the second sum (and all sums over $S$ for the remainder of this proof), $S$ ranges over all $(k-1)$-subsets of $V(G)$.
Note that $p_{v}(S)=0$ when $\{v\}\cup S$ is not an edge of $G$,
so the sum in \cref{eq:technical-hope} can be viewed as a sum over the $k|E(G)|$ choices of $v,S$
such that $\{v\}\cup S$ is an edge. 
Order these $(v,S)$ pairs arbitrarily so that each $1\leq i\leq k|E(G)|$ corresponds to some $(v,S)$ pair.
We will prove \cref{eq:technical-hope} by verifying the conditions in \cref{lem:technical-kahn}, with $t = k$, $l=k|E(G)|$, $D= l/n$, $p_i = p_v(S)$,  $\gamma_i = \gamma_v(S)$, some $\varepsilon=o(1)$ and some $\lambda=O(1)$.

First, note that $\sum_{S}p_{v}(S)=1$
by the definition of $p_{v}(S)$, so $\sum_{v}\sum_Sp_{v}(S)=n$.

Second, writing $D=k|E(G)|/n=\Theta(n^{k-1})$, we also have
\[
\sum_{v\in V(G)}\sum_{S}p_{v}(S)\ln\left(\frac{1}{p_{v}(S)}\right) = kh(\xx) \geq k\ln\left(\Phi(G)\right)\ge(k-1)n\ln n-O(n)=n\ln(D)-O(n)\,.
\]
Indeed, in the first inequality, viewing $\mathbf f$ as a random vector on $\{0,1\}^{E(G)}$, we have
\[
    \ln (\Phi(G)) = H(\textbf{f}) = H(\mathbbm{1}_{e_1\in \mathbf f}, \cdots,\mathbbm{1}_{e_{|E(G)|}\in \mathbf f}) \leq \sum_{e\in E(G)} H(\mathbbm{1}_{e\in \mathbf f}) =  \frac{1}{k}\sum_{v\in V(G)}\sum_{e\ni v} \xx[e]\ln(1/\xx[e])= h(\xx)\,,
\]
which follows from \Cref{cla:drop-conditioning} and \Cref{cla:chain-rule}; in the last inequality, we are using a crude lower bound on $\Phi(G)$ which appears
for example as \cite[Corollary 1.7(i)]{spread-measure} (it can also be easily
deduced from results in \cite{PSS23} or \cite{kelly2023optimal}, or from the lower bound in \cref{thm:num-of-PM} proved in \cref{section:random-greedy} together with \cref{cla:jensen-entropy}). 
Thus, we may take $\lambda = O(1)$.

Then, for any vertex $v$, perfect matching $f$, 
and $(k-1)$-set of vertices $S$, note that $f\in\Gamma_{v}(S)$ if and only if
$\{v\}\cup S$ intersects some edge of $f$ in more than
one vertex. So, for each perfect matching $f$, 
\[
\left|\left\{ (v,S):f\in\Gamma_{v}(S)\right\} \right|\le\frac{n}{k}\binom{k}{2}\binom{n-2}{k-2}k=o(n^{k})\,.
\]
Indeed, there are at most $\frac{n}{k}\binom{k}{2}\binom{n-2}{k-2}$ sets of $k$ vertices which intersect an edge of $f$ in more than one vertex, and there are at most $k$ ways to represent this set of $k$ vertices as $\{v\}\cup S$.
It follows that
\[
\sum_{v\in V(G)}\sum_S\gamma_{v}(S)=\sum_{f}\Pr[\mathbf{f}=f]\cdot\left|\left\{ (v,S):f\in\Gamma_{v}(S)\right\} \right|=o(n^{k})=o(nD)\,,
\]
where the sum on the right-hand side ranges over all perfect matchings $f$ of $G$.
Thus, we may take $\epsilon = o(1)$.
So, \cref{eq:technical-hope} follows from \cref{lem:technical-kahn}.
\end{proof}

\section{Lower bound on \texorpdfstring{\(h(G)\)}{h(G)} when \texorpdfstring{\(d\ge k/2\)}{d≥k/2}}
\label{section:lower-bound-for-upper range}
The goal of this section is to give a tight lower bound on the entropy of a $k$-graph with a given minimum $d$-degree, when $d\ge k/2$. 
Specifically, we will prove the following theorem.
\begin{theorem}\label{thm:lb-entropy}
    For $d,k\in \mathbb N$ satisfying $k/2\leq d\leq k-1$ and any constant $\gamma > 0$, and let $G$ be a $(d,\gamma)$-Dirac $k$-graph on $n$ vertices.
    Then,
    \begin{equation}\label{eqn:lb-entropy}
        h(G) \geq \frac{n}{k} \ln\left(\frac{k}{n}\cdot\frac{\binom{n}{d}}{\binom{k}{d}}\cdot\delta_{d}(G)\right)\,.
    \end{equation}
\end{theorem}
Note that when $G = K_n^{(k)}$ is the complete $k$-graph on $n$ vertices, this lower bound exactly matches the upper bound given by Jensen's inequality (see \Cref{cla:jensen-entropy}).
We first give the (quick) deduction of \Cref{thm:lb-no-entropy} from \Cref{thm:lb-entropy} and \cref{thm:num-of-PM}.
\begin{proof}[Proof of \Cref{thm:lb-no-entropy} given \Cref{thm:lb-entropy}]
Our goal is to show that
\[\ln \Phi(G) \geq \ln(\Phi_n^{(k)}) + \frac{n}{k} \ln p + o(n)\,.\]
From \Cref{thm:num-of-PM}, we know that 
\[
    \ln \Phi(G) \geq h(G) - (1-1/k)n + o(n)\,.
\]

Applying \Cref{thm:lb-entropy} and recalling $p=\delta_{d}(G)/\binom{n-d}{k-d}$, and noting that $\ln \binom nk=k\ln n+o(1)-\ln(k!)$, we deduce
\begin{align*}
    \ln \Phi(G) &\geq \frac{n}{k} \ln\left(\frac{k}{n}\cdot\frac{\binom{n}{d}}{\binom{k}{d}}\cdot\binom{n-d}{k-d}\cdot p\right) - \left(1-\frac 1 k\right)n + o(n)\\
    &= \frac{n}{k}\ln\biggl(\frac kn\cdot \binom{n}{k}p\biggr) - \left(1-\frac 1 k\right)n + o(n)\\
    &= \frac nk\ln\left(\frac kn\right)+n\ln n-\frac nk\ln(k!)+\frac{n}{k} \ln p - \left(1-\frac 1 k\right)n+ o(n)\,.
\end{align*}
On the other hand, we have \begin{align*}\ln\left(\Phi(K_n^{(k)})\right)=\ln\left(\frac{n!}{(n/k)!\cdot (k!)^{n/k}}\right)&=n\ln \left(\frac n e\right)-\frac n k\ln\left(\frac{n/k}e\right)-\frac nk\ln(k!)\\
&=n\ln n+\frac nk\ln\left(\frac kn\right)-\left(1-\frac 1 k\right)n-\frac nk\ln(k!)\end{align*}
by Stirling's approximation; the desired result follows.
\end{proof}

We will prove \Cref{thm:lb-entropy} by constructing an auxiliary bipartite graph $G$ from our hypergraph $H$, and applying the following result of Cuckler and Kahn \cite[Theorem 3.1]{CK-counting} (proved via a linear-algebraic construction). 

\begin{lemma}\label{lem:non-negative-bipartite}
    Let $G$ be a balanced bipartite graph on $2n$ vertices with minimum degree $\delta(G) \geq n/2$. Then, $h(G) \geq n\ln(\delta(G))$. 
\end{lemma}
\begin{proof}[Proof of \Cref{thm:lb-entropy}]
    Using the $k$-graph $G$, we construct a bipartite graph $\tilde{G}$, with parts $\tilde A$ and $\tilde B$, as follows. First, $\tilde A$ consists of all subsets of $d$ vertices of $V(G)$, each duplicated exactly $\binom{n}{k-d}$ times, and $\tilde B$ consists of all subsets of $(k-d)$ vertices of $V(G)$, each duplicated $\binom{n}{d}$ times. This means that
    \[|\tilde A|=|\tilde B|= \binom{n}{d}\binom{n}{k-d}\,.\]
    Then, in $\tilde G$ we put an edge between $U\in \tilde A$ and $W\in \tilde B$ whenever $U\cup W$ forms an edge in $G$.
    Thus, each edge $e\in E(G)$ leads to $\binom{k}{d}\binom{n}{d}\binom{n}{k-d}$ corresponding edges in $E(\tilde{G})$. We write $\tilde{e} \sim e$ to denote that an edge $\tilde{e}\in E(\tilde G)$ arose from an edge $e\in E(G)$.

    Since $\tilde B$ has $\binom nd$ copies of each set of $k-d$ vertices, in the graph $\tilde G$ each $U\in \tilde A$ has degree at least $\binom{n}{d}\delta_d(G)$, and each $W\in \tilde B$ has degree at least $\binom{n}{k-d}\delta_{k-d}(G)$. Since $d\geq k/2$, \cref{cla:monotone} implies that $\binom{n}{d}\delta_d(G) \leq \binom{n}{k-d}\delta_{k-d}(G)$, and therefore 
    \begin{equation}\delta(\tilde{G}) = \binom{n}{d}\delta_d(G)\,.\label{eq:min-degree-translation}\end{equation}
Using \cref{fact:dirac-bound}, we compute
     \[
        \delta(\tilde{G})
        = \binom{n}{d}\delta_d(G) \geq (1/2+\gamma)\binom{n}{d}\binom{n-d}{k-d} 
        \geq (1/2)\binom{n}{d}\binom{n}{k-d} \ge \tilde n/2\,,
    \]
    writing $\tilde n=|\tilde A|=|\tilde B|$. That is to say, $\tilde{G}$ is a balanced bipartite graph on $2\tilde n$ vertices with minimum degree at least $\tilde n/2$.
    
    By \Cref{lem:non-negative-bipartite}, there exists a fractional perfect matching $\tilde{\xx}$ in $\tilde{G}$ satisfying
    \begin{equation}\label{eqn:large-entropy}
          \sum_{e\in E(\tilde{G})}\tilde{\xx}[e] \ln(1/\tilde{\xx}[e]) \ge \tilde n\ln(\delta(\tilde G))\,.
    \end{equation}
    
    Now, define
    \[L\coloneqq \frac{k}{n}\binom{n}{d}\binom{n}{k-d},\quad \xx[e] = \sum_{\tilde{e}\sim e} \tilde{\xx}[\tilde{e}]/L\text{ for each edge }e\in E(G)\,.\]
    We will show $\xx\in \R^m$ is a fractional perfect matching of $G$, certifying \cref{eqn:lb-entropy}.

First, it is a straightforward calculation to see that $\xx$ is a fractional perfect matching: certainly each $\xx[e]$ is nonnegative, and for every $v\in V(G)$ we have
    \begin{align*}
        \sum_{e\ni v}\xx[e] 
        &= \frac{1}{L}\sum_{e\ni v}\sum_{\tilde{e}\sim e} \tilde{\xx}[\tilde{e}] 
        = \frac{1}{L} \left(\sum_{\substack{U\in \tilde A: \\ v\in U}} \sum_{\substack{\tilde{e}\in E(\tilde{G}):\\ U\in \tilde{e}}} \tilde{\xx}[\tilde{e}] + \sum_{\substack{W\in\tilde B: \\ v\in W}}\sum_{\substack{\tilde{e}\in E(\tilde{G}):\\ W\in \tilde{e}}} \tilde{\xx}[\tilde{e}] \right)
        = \frac{1}{L}\left(\sum_{\substack{U\in \tilde A: \\ v\in U}} 1 + \sum_{\substack{W\in\tilde B: \\ v\in W}} 1\right)\\
        &= \frac{1}{L} \biggl(\binom{n}{k-d}\binom{n-1}{d-1} + \binom{n}{d}\binom{n-1}{k-d-1}\biggr) = 1\,,
    \end{align*}
    recalling that $\tilde A$ contains $\binom n{k-d}$ copies of each $d$-set of vertices, and $\tilde B$ contains $\binom nd$ copies of each $(k-d)$-set of vertices.
    It remains to lower bound the entropy of $\xx$.

Let $Q=\binom{k}{d}\binom{n}{d}\binom{n}{k-d}=\binom{k}{d}\tilde n$ be the number of edges
$\tilde{e}\in E(G)$ arising from each $e\in E(G)$. We compute
\begin{align*}
h(\xx) = \sum_{e\in E(G)}\xx[e]\ln(1/\xx[e]) &=  \sum_{e\in E(G)} \left(\frac{1}{L}\sum_{\tilde{e}\sim e}\tilde{\xx}[\tilde{e}]\right) \ln \left(\frac{L}{\sum_{\tilde{e}\sim e}\tilde{\xx}[\tilde{e}]}\right)\\
 & =\frac{\ln(L/Q)}{L}\sum_{e\in E(G)}\sum_{\tilde{e}\sim e}\tilde{x}[\tilde{e}]+\frac{Q}{L}\sum_{e\in E(G)}\left(\frac{\sum_{\tilde{e}\sim e}\tilde{x}[\tilde{e}]}{Q}\right)\ln\left(\frac{Q}{\sum_{\tilde{e}\sim e}\tilde{x}[\tilde{e}]}\right)\\
 & \ge\frac1 L\ln\left(\frac kn\cdot \frac 1{\binom kd}\right)\sum_{\tilde{e}\in E(\tilde{G})}\tilde{x}[\tilde{e}]+\frac{Q}{L}\sum_{e\in E(G)}\sum_{\tilde{e}\sim e}\frac{1}{Q}\tilde{x}[\tilde{e}]\ln\left(\frac{1}{\tilde{x}[\tilde{e}]}\right)\\
 & \ge\frac1 L\ln\left(\frac kn\cdot \frac 1{\binom kd}\right)\tilde{n}+\frac{1}{L}\tilde{n}\ln(\delta(\tilde G))\\
 & \ge \frac{n}{k}\ln\left(\frac{k}{n}\cdot\frac{\binom{n}{d}}{\binom{k}{d}}\cdot\delta_{d}(G)\right)\,,
\end{align*}
where for the third line we used Jensen's inequality with the concave function $t\mapsto t\log(1/t)$, for the
fourth line we used \cref{eqn:large-entropy}, and for the last line we used \cref{eq:min-degree-translation}.
\end{proof}


\begin{thebibliography}{}
\bibitem{AFHRRS}
N.~Alon, P.~Frankl, H.~Huang, V.~R\"{o}dl, A.~Ruci\'{n}ski, and B.~Sudakov, \emph{Large matchings in uniform hypergraphs and the conjecture of {E}rd{\H{o}}s and {S}amuels}, J. Combin. Theory Ser. A \textbf{119} (2012), no.~6, 1200--1215.

\bibitem{BFL10}
T.~Bohman, A.~Frieze, and E.~Lubetzky, \emph{A note on the random greedy triangle-packing algorithm}, J. Comb. \textbf{1} (2010), no.~3-4, 477--488.

\bibitem{BFL15}
T.~Bohman, A.~Frieze, and E.~Lubetzky, \emph{Random triangle removal}, Adv. Math. \textbf{280} (2015), 379--438.

\bibitem{Bondy-book}
J.~A. Bondy, \emph{Basic graph theory: paths and circuits}, Handbook of combinatorics, {V}ol. 1, 2, Elsevier Sci. B. V., Amsterdam, 1995, pp.~3--110.

\bibitem{CT-entropy}
T.~M. Cover and J.~A. Thomas, \emph{Elements of information theory}, Wiley Series in Telecommunications, John Wiley \& Sons, Inc., New York, 1991, A Wiley-Interscience Publication.

\bibitem{CK-entropy}
B.~Cuckler and J.~Kahn, \emph{Entropy bounds for perfect matchings and {H}amiltonian cycles}, Combinatorica \textbf{29} (2009), no.~3, 327--335.

\bibitem{CK-counting}
B.~Cuckler and J.~Kahn, \emph{Hamiltonian cycles in {D}irac graphs}, Combinatorica \textbf{29} (2009), no.~3, 299--326.

\bibitem{Ferber-counting}
A.~Ferber, L.~Hardiman, and A.~Mond, \emph{Counting {H}amilton cycles in {D}irac hypergraphs}, Combinatorica \textbf{43} (2023), no.~4, 665--680.

\bibitem{FKS-counting-hypergraph}
A.~Ferber, M.~Krivelevich, and B.~Sudakov, \emph{Counting and packing {H}amilton {$\ell$}-cycles in dense hypergraphs}, J. Comb. \textbf{7} (2016), no.~1, 135--157.

\bibitem{FM-Dirac-threshold}
A.~Ferber and M.~Kwan, \emph{Dirac-type theorems in random hypergraphs}, J. Combin. Theory Ser. B \textbf{155} (2022), 318--357.

\bibitem{Erdos-matching-conj}
P.~Frankl and A.~Kupavskii, \emph{The {E}rd{\H{o}}s matching conjecture and concentration inequalities}, J. Combin. Theory Ser. B \textbf{157} (2022), 366--400.

\bibitem{Fre75}
D.~A. Freedman, \emph{On tail probabilities for martingales}, Ann. Probability \textbf{3} (1975), 100--118.

\bibitem{Stefan-counting}
S.~Glock, S.~Gould, F.~Joos, D.~K\"{u}hn, and D.~Osthus, \emph{Counting {H}amilton cycles in {D}irac hypergraphs}, Combin. Probab. Comput. \textbf{30} (2021), no.~4, 631--653.

\bibitem{Absorbing}
H.~H\`{a}n, Y.~Person, and M.~Schacht, \emph{On perfect matchings in uniform hypergraphs with large minimum vertex degree}, SIAM J. Discrete Math. \textbf{23} (2009), no.~2, 732--748.

\bibitem{Kahn-Shamir}
J.~Kahn, \emph{Asymptotics for {S}hamir's problem}, Adv. Math. \textbf{422} (2023), Paper No. 109019, 39.

\bibitem{spread-measure}
D.~Y. Kang, T.~Kelly, D.~K{\"u}hn, D.~Osthus, and V.~Pfenninger, \emph{Perfect matchings in random sparsifications of dirac hypergraphs}, Combinatorica (2024).

\bibitem{kelly2023optimal}
T.~Kelly, A.~Müyesser, and A.~Pokrovskiy, \emph{Optimal spread for spanning subgraphs of dirac hypergraphs}, 2023.

\bibitem{PSS23}
H.~T. Pham, A.~Sah, M.~Sawhney, and M.~Simkin, \emph{A toolkit for robust thresholds}, 2023.

\bibitem{RR-survey}
V.~R\"{o}dl and A.~Ruci\'{n}ski, \emph{Dirac-type questions for hypergraphs---a survey (or more problems for {E}ndre to solve)}, An irregular mind, Bolyai Soc. Math. Stud., vol.~21, J\'{a}nos Bolyai Math. Soc., Budapest, 2010, pp.~561--590.

\bibitem{RRS09}
V.~R\"{o}dl, A.~Ruci\'{n}ski, and E.~Szemer\'{e}di, \emph{Perfect matchings in large uniform hypergraphs with large minimum collective degree}, J. Combin. Theory Ser. A \textbf{116} (2009), no.~3, 613--636.

\bibitem{SSS}
G.~N. S\'{a}rk\"{o}zy, S.~M. Selkow, and E.~Szemer\'{e}di, \emph{On the number of {H}amiltonian cycles in {D}irac graphs}, Discrete Math. \textbf{265} (2003), no.~1-3, 237--250.

\bibitem{TZ12}
A.~Treglown and Y.~Zhao, \emph{Exact minimum degree thresholds for perfect matchings in uniform hypergraphs}, J. Combin. Theory Ser. A \textbf{119} (2012), no.~7, 1500--1522.

\bibitem{TZ13}
A.~Treglown and Y.~Zhao, \emph{Exact minimum degree thresholds for perfect matchings in uniform hypergraphs {II}}, J. Combin. Theory Ser. A \textbf{120} (2013), no.~7, 1463--1482.

\bibitem{Zhao-survey}
Y.~Zhao, \emph{Recent advances on {D}irac-type problems for hypergraphs}, Recent trends in combinatorics, IMA Vol. Math. Appl., vol. 159, Springer, [Cham], 2016, pp.~145--165.

\bibitem{Count-oriented-trees}
Felix Joos and Jonathan Schrodt.
\newblock Counting oriented trees in digraphs with large minimum semidegree.
\newblock {\em J. Combin. Theory Ser. B}, 168:236--270, 2024.

\end{thebibliography}
\end{document}